\documentclass[10pt, a4paper]{amsart}

\usepackage{lmodern}
\usepackage{microtype}
\usepackage{amssymb}
\usepackage{mathtools} % patches some things in amsmath, loads it anyway
\usepackage[backend=bibtex8,firstinits, arxiv=abs, doi = false, url = false, style = numeric,
			isbn=false, backrefstyle=three, hyperref]{biblatex} 
\renewbibmacro{in:}{\ifentrytype{article}{}{\printtext{\bibstring{in}\intitlepunct}}} %Remove 'In:' from Journal entries.

\usepackage[english]{babel}
\usepackage{csquotes}
\usepackage{mleftright} % better \left \right accessed by \mleft and \mright
\usepackage{graphicx}
\usepackage{url}
\usepackage{amscd}
\usepackage[textsize=footnotesize]{todonotes} % to keep track of things that need to be completed, \todo

\usepackage{thmtools} % more features on declaretheorem
\usepackage{hyperref}
\usepackage{cleveref} % automatically include reference type, \Cref{sec:intro} to get `Section~1'

\hypersetup{plainpages = false,
    pdfpagelayout = SinglePage, 
    bookmarksopen = true, 
    bookmarksnumbered = true, linktocpage,
    colorlinks = true, 
    linkcolor = black, % internal links being too prominent increases clutter
    urlcolor  = blue,
    citecolor = red,  anchorcolor = green}

\DeclareGraphicsExtensions{.png, .jpg, .pdf}
\graphicspath{{img/}{img/}{img/}}

\usepackage[a4paper, centering, scale=0.77]{geometry} % the modern way of dealing with page geometry, without pure tex setlength's scale=0.75,marginpar=1in

\title[Reflections on manifolds]%
			{Coxeter transformation groups and reflection arrangements in smooth manifolds}
\author[R\, Das]{Ronno Das}
\author[P\, Deshpande ]{Priyavrat Deshpande}
\thanks{Both the authors are partially funded by a grant from Infosys Foundation}
\address{Chennai Mathematical Institute \\ H1 SIPCOT IT Park, Siruseri \\ Tamil Nadu, India}
\email{ronno@cmi.ac.in}
\email{pdeshpande@cmi.ac.in} 


\makeatletter
\def\thmt@refnamewithcomma #1#2#3,#4,#5\@nil{%
  \@xa\def\csname\thmt@envname #1utorefname\endcsname{#3}%
  \ifcsname #2refname\endcsname
    \csname #2refname\expandafter\endcsname\expandafter{\thmt@envname}{#3}{#4}%
  \fi
}
\makeatother
\declaretheorem[parent=section,refname={theorem,theorems},Refname={Theorem,Theorems}]{theorem}
\declaretheorem[sibling=theorem,style=definition,numbered=yes,refname={definition,definitions},Refname={Definition,Definitions}]{definition}
\declaretheorem[sibling=theorem,refname={lemma,lemmas},Refname={Lemma,Lemmas}]{lemma}

\declaretheorem[sibling=theorem,refname={corollary,corollaries},Refname={Corollary,Corollaries}]{corollary}
\declaretheorem[sibling=theorem, refname={proposition,propositions}, Refname={Proposition,Propositions}, name=Proposition]{proposition}
\declaretheorem[numbered=no,style=remark,refname={remark,remarks},Refname={Remark,Remarks}]{remark}
\declaretheorem[parent=section, style=remark, refname={example,examples}, Refname={Example,Examples}, name=Example]{example}

\numberwithin{equation}{section}

\newcommand{\bt}[1]{\begin{theorem}\label{#1}}
\newcommand{\bc}[1]{\begin{corollary}\label{#1}}
\newcommand{\bl}[1]{\begin{lemma}\label{#1}}
\newcommand{\bp}[1]{\begin{proposition}\label{#1}}
\newcommand{\be}[1]{\begin{example}\label{#1}}
\newcommand{\bd}[1]{\begin{definition}\label{#1}}
\newcommand{\br}[1]{\begin{remark}\label{#1}}
\newcommand{\bx}[1]{\begin{exercise}\label{#1}}
\newcommand{\bcon}[1]{\begin{conjecture}\label{#1}}

\newcommand{\et}{\end{theorem}}
\newcommand{\ec}{\end{corollary}}
\newcommand{\el}{\end{lemma}}
\newcommand{\ep}{\end{proposition}}
\newcommand{\ee}{\end{example}}
\newcommand{\ed}{\end{definition}}
\newcommand{\exc}{\end{exercise}}
\newcommand{\er}{\end{remark}}
\newcommand{\econ}{\end{conjecture}}

\newcommand{\bpr}{\begin{proof}}
\newcommand{\epr}{\end{proof}}

\def\A  {\mathcal{A}}
\def\FA {\mathcal{F(A)}}
\def\Ch {\mathcal{C(\A)}}

\def \F {\mathcal{F}}

\def \RR {\mathcal{R}}
\def \P {\mathcal{P}}

\def \s {\mathcal{S}}
\def \L {\mathcal{L}}

\def\R{\mathbb{R}}

\def\C{\mathbb{C}}
\def\Z{\mathbb{Z}}

\def\ts{\textstyle}

\newcommand*{\from}{\vcentcolon} % makes function definitions semantically better 
\newcommand*{\set}[2]{\mleft\{#1 \;\middle|\; #2\mright\}} % set comprehensions with auto sized delimiters
\DeclareMathOperator{\Sal}{Sal}
\DeclareMathOperator{\Chain}{Chain}
\DeclareMathOperator{\Lk}{Lk}
\newcommand{\ol}{\overline} % I have a more general closure macro, but this is probably fine
\newcommand*{\dispunct}[1]{\,\text{#1}} % puncutation at the end of display math
\begin{document}
\begin{abstract}
Artin groups are a natural generalization of braid groups and are well-understood in certain cases.
Artin groups are closely related to Coxeter groups in the following sense.
There is a faithful representation of a Coxeter group $W$ as a linear reflection group on a real vector space.
The group acts properly and fixes a union of hyperplanes.
The $W$-action extends as the covering space action to the complexified complement of these hyperplanes.
The fundamental groups of the complement and that of the orbit space are the pure Artin group and the Artin group respectively.
For the Artin groups of finite type Deligne proved that the associated complement is aspherical.
Using the Coxeter group data Salvetti gave a construction of a cell complex which is a $W$-equivariant deformation retract of the complement.
This construction was independently generalized by Charney and Davis to the Artin groups of infinite type.
A lot of algebraic properties of these groups were discovered using combinatorial and topological aspects of this cell complex.
 
In this paper we represent a Coxeter group as a subgroup of diffeomorphisms of a smooth manifold. 
These so-called Coxeter transformation groups fix a union of codimension-$1$ (reflecting) submanifolds and permute the connected components of the complement.
Their action naturally extends to the tangent bundle of the ambient manifold and fixes the union of tangent bundles of these reflecting submanifolds.
Fundamental group of the tangent bundle complement and that of its quotient serve as the analogue of pure Artin group and Artin group respectively.
The main aim of this paper is to prove Salvetti's theorems in this context.
We show that the combinatorial data of the Coxeter transformation group can be used to construct a cell complex which is equivariantly homotopy equivalent to the tangent bundle complement.

\end{abstract}
\maketitle
% Introduction
\section{Introduction}\label{intro}
A Coxeter group $W$ is a group with the following presentation - 
\begin{equation}
W = \langle s_1, \dots, s_n \mid s_i^2 = 1, (s_is_j)^{m_{ij}} = 1,\quad \forall i\neq j \text{~and~} 2\leq m_{ij} \leq \infty\rangle\dispunct{.} \label{coxeq}
\end{equation}
Let us first consider the case when $W$ is finite and irreducible. 
Such a group acts linearly (via origin-fixing isometries) on a real vector space $V$ of dimension $n$. 
Its action on $V$ is not free; each reflection (i.e., conjugate of a generating reflection) in $W$ fixes a hyperplane. 
The union of these reflecting hyperplanes is the \emph{reflection arrangement}, denoted $\A_W$, associated to $W$. 
The arrangement induces a stratification of the ambient space into convex open subsets called \emph{faces}.
The codimension-$0$ faces are open simplicial cones called \emph{chambers} and they are permuted freely
under the $W$ action (see \cite[Chapter 6]{davisbook08} and \cite[Chapter 5]{humph90} for details). 

Complexifying this situation we get a finite arrangement of complex hyperplanes in $V\otimes\C$. 
The complement of these hyperplanes, denoted $M_W$, is connected and admits a fixed point free action of $W$. 
Brieskorn \cite{bries73} proved that the fundamental group of the orbit space $N_W := M_W/W$ has the following presentation -
\begin{equation}
\langle s_1,\dots, s_n \mid \underbrace{s_is_js_i\cdots}_{m_{ij}} = \underbrace{s_js_is_j\cdots}_{m_{ij}}, \quad \forall i\neq j\rangle\dispunct{,} \label{artineq}
\end{equation}
where $m_{ij}$'s are integers $\ge 2$.
This is the \emph{Artin group (of finite type)} associated to $W$ and is denoted $A_W$. 
There is a natural surjection from $A_W$ onto $W$ with the kernel the so-called \emph{pure Artin group} $PA_W$; it is the fundamental group of $M_W$. 
For example, if $W$ is the symmetric group on $n$ letters then $A_W$ is the braid group on $n$ strands and $PA_W$ is the pure (or the coloured) braid group.

In his seminal paper Deligne \cite{deli72} showed that the universal cover of $N_W$ is contractible, i.e., $N_W$ is a $K(A_W, 1)$ space. 
Subsequent study of Artin groups is much influenced by Deligne's work. 
Some of the important properties of Artin groups were proved by expanding on his ideas, notably the biautomatic nature of these groups \cite{charney92}. 
Simply put, it says that Artin groups have solvable word and conjugacy problems. \par 
%We refer the reader to \cite[Section 1.2]{charney07} for details and references.
The key tool in dealing with $M_W$ or $N_W$ is the Salvetti complex.
A pioneering result by Salvetti in \cite{sal1} states that the homotopy type of the complexified complement of a locally finite hyperplane arrangement is determined by the incidence relations among the faces of the arrangement.
He further showed in \cite{salvetti94} that for finite reflection arrangements the cell complexes associated to $M_W$ and $N_W$ can be constructed using the combinatorial data of the Coxeter group.\par 
These complexes are constructed as follows. 
First note that the faces of a reflection arrangement $\A_W$ are in one-to-one correspondence with the conjugates of the standard parabolic subgroups of $W$ \cite[Chapter 5]{humph90}.
In fact, the set of all parabolic subgroups $\set{wW_T}{w\in W, T\subseteq S}$ is a poset with the partial order given by $w'W_{T'} \le wW_{T}$ if and only if $T'\subseteq T$ and $w^{-1}w'\in W_T$.
The geometric realization of this poset is the \emph{Coxeter complex} of $\A_W$.

Now, consider the finite collection $\set{(w, T)}{w\in W, T\subseteq S}$ with the partial order given by $(w, T) < (w', T')$ if and only if $T' \subseteq T$ and for every $t\in T$ we have $\ell(w^{-1}w') \prec \ell(tw^{-1}w')$ where $\ell$ counts the length of $w$ as a word over $S$.
Let $\Sal(W)$ denote the regular cell complex whose face poset is isomorphic to the poset constructed above. %Why aren't we using the term `geometric realization'? It is not clear that there is only one regular cell complex with a given face poset.
Then Salvetti showed that $\Sal(W)$ is a $W$-equivariant deformation retract of $M_W$.
The natural action of $W$ on $\Sal(W)$ is given by $u\cdot (w, T) = (uw, T)$. 
Let $\ol{\Sal}(W)$ denote the orbit space of this action; it has the same homotopy type as $N_W$.
There is a unique vertex in $\ol{\Sal}(W)$ denoted by $[w, \emptyset]$.
In fact, for every subset $T\subseteq S$ of cardinality $k$ there is a unique $k$-cell indexed by $[w, T]$.

The story is not very different for infinite Coxeter groups. 
An infinite Coxeter group $W$ acts faithfully on a convex cone with non-empty interior in $V$. 
On the interior, which is known as the \emph{Tits cone} and denoted $I$, the action of $W$ is proper. 
The fixed-point set of $W$ is a union of (possibly infinite) family of linear hyperplanes in $V$ such that  for each fixed hyperplane $H$ the intersection $H\cap I$ is non-empty and for every $x\in I$ there exists an open neighbourhood $U_x$ which intersects only finitely many fixed hyperplanes. 
Denote by $M_W$ the following analogue of the complexified complement
\begin{equation}\label{vinmw}
M_W = \left( \big(I\times I\big) \setminus \bigg(\bigcup_{H\in\A} (H\times H)\bigg)\right)
\end{equation}
we call this the \emph{Vinberg complement}.
As proved in the thesis of van der Lek \cite{lek83} the fundamental group of the orbit space $N_W := M_W/ W$
is the \emph{infinite-type Artin group} associated to $W$.

In general it is not known whether the infinite-type Artin groups are torsion free, have solvable word or conjugacy problems. Moreover the analogue of Deligne's theorem is also missing.
The Salvetti complex construction does extend to these groups with a slight modification; one has to consider only those subsets $T$ of $S$ for which the subgroup generated by reflections in $T$ is finite (see Charney-Davis \cite{charneyDavis2}).
These are called the \emph{spherical subsets} of $S$.
The $K(\pi, 1)$ conjecture states that this Salvetti complex is a model for the classifying space of the pure Artin group. 
%The Salvetti complex for infinite-type Artin groups is very similar to the one described above.
%The difference is that one takes only those subsets of .
For certain sub-classes of infinite-type Artin groups, like FC-type \cite{charneyDavis1}, and certain affine Artin groups it has been proved that the corresponding hyperplane complement is aspherical (see \cite[Page 9]{parisarxiv} for an exhaustive list). % change to detailed? 

A generalization of hyperplane arrangements was introduced by the second author in \cite{deshpande_thesis11} where a study of arrangements of codimension-$1$ submanifolds in a smooth manifold was initiated.
%Recall that finite reflection groups are discrete subgroups of isometries of a sphere of appropriate dimension \cite[Section 6.7]{davisbook08}.
%The fixed point set of the action of a reflection group on a sphere is an arrangement of subspheres of codimension-$1$ and the chambers are simplices.
A Coxeter group can be represented as a group of diffeomorphisms of a smooth manifold.
In fact it is well-known that a discrete group of diffeomorphisms generated by finitely many (dissecting) reflections has a Coxeter presentation as given in \Cref{coxeq} (see \cite{michoretal07, gutkin86, straume81}).
%Hence they are also called \emph{Coxeter transformation groups}.
The action of such a transformation group on a smooth manifold fixes a union of codimension-$1$ submanifolds which we call the \emph{manifold reflection arrangement}.
To such a reflection arrangement we associate a topological space called the tangent bundle complement.
This serves as a natural generalization of the complexified complement since topologically $\C^n$ is the tangent bundle of $\R^n$.
The manifold reflection group acts freely and properly discontinuously on the tangent bundle complement.
The fundamental group of the complement and that of the orbit space serves as an analogue of the pure Artin group and the Artin group respectively.
The main aim of this paper is to lay topological foundations for the study of these ``Artin-like'' groups.  

The paper is organized as follows. %\todo{Not happy with this para.}
In \Cref{sec:manrefl} we recall the relevant material regarding the reflections on smooth manifolds and discrete subgroups of diffeomorphisms generated by such reflections.
The main purpose of this section is to present a comparison between  manifold reflection groups and the classical reflection groups.
In \Cref{sec:nerve lemma} we first define manifold reflection arrangements and explore their combinatorial properties. 
Then we introduce the idea of the tangent bundle complement and prove the main theorem of this paper (\Cref{main theorem}).
The result says that there exists a regular cell complex defined using the combinatorics of the reflection arrangement with the same homotopy type as that of the tangent bundle complement.
This theorem generalizes the classical theorems of Salvetti. Our proof is based on the one given by Paris in his survey of the $K(\pi, 1)$ conjecture \cite{parisarxiv}.
Finally, in \Cref{sec:group} we show how to use the group theoretic information to describe this Salvetti complex.
We also give a presentation for the fundamental group of the quotient complex which turns out to be similar to that of an Artin group in most cases.

% Section 2: Coxeter transformation groups
\section{Coxeter transformation groups}\label{sec:manrefl}
In this section we present a brief review of manifold reflection groups.
These are subgroups of the diffeomorphisms group and are generated by finitely many reflections. 
It is interesting to note that the combinatorial and geometric properties of reflection groups are independent of the particular nature of the group action. 
For example, under reasonable (topological) assumptions, these manifold reflection groups have Coxeter presentation.
Also the fixed point set behaves very much like the reflection arrangement.
Moreover the combinatorics of the induced stratification is independent of the topology of the manifold.
The main references for this section are Gutkin \cite{gutkin86}, Davis \cite[Chapter 10]{davisbook08} and Alekseevsky et al \cite{michoretal07} (they consider subgroups of isometries of a complete Riemannian manifold).
See also Straume \cite{straume81} for similar results for groups generated by continuous reflections. 

We begin by recalling some basics of group actions on manifolds, the main reference being Bredon's book \cite{bredon72}.
Let $X$ be a smooth, connected manifold without boundary of dimension $l$ and let $G$ be a discrete group of diffeomorphisms. % Shouldn't we use W throughout?
We assume that the $G$-action on $X$ is smooth and proper (i.e., properly discontinuous).
In case the manifold $X$ is open and $G$ is infinite then we prefer that $G$ acts co-compactly.  

As the group action is proper, for every $x\in X$ the associated isotropy subgroup $G_x$ is finite. 
There exists an open neighbourhood $U_x$ of $x$ such that $gU_x\cap U_x = \emptyset$ for $g\in G\setminus G_x$. 
The orbit space $X/G$ is Hausdorff. 
Each orbit $G(x)$ is closed and discrete and by the orbit-stabilizer theorem it is in bijection with $G/G_x$.

Let $H$ be a subgroup of $G$ and let $A$ be a left $H$-space. 
Then $H$ acts on $G\times A$ via $h\cdot(g, a) = (gh^{-1}, ha)$. 
The \emph{twisted product} of $G$ and $A$ is the orbit space of the $H$ action on $G\times A$. 
The twisted product is denoted by $G\times_H A$ and we write $[g, x]$ for the equivalence class of $(g, x)$.
Recall that a \emph{$G$-tube} about an orbit $P$ is a $G$-equivariant embedding $\phi\colon G\times_H A\to X$ onto an open neighbourhood of $P$ in $X$. 
If $A$ is homeomorphic to a disc, then the tube is an \emph{equivariant tubular neighbourhood} of $P$.

\bd{def2s1}Let $x\in S\subset X$ be such that $G_x(S) = S$. Then $S$ is called a \emph{slice} at $x$, if the map 
\[G\times_{G_x} S \to X\]
taking $[g, s]$ to $gs$ is a $G$-tube about $G(x)$. 
A slice $S$ is \emph{linear} if it is $G_x$-equivariantly homeomorphic to a neighbourhood of the origin in some linear representation of $G_x$ on $\R^l$. \ed

A group action on a smooth manifold is \emph{locally smooth} (or locally linear) if every point has a linear slice. 
If a discrete group acts properly as well as smoothly on a manifold $X$ then the action is locally smooth. 
Since each isotropy subgroup $G_x$ is finite there is a $G_x$-invariant metric on $X$ (see \cite[Chapter VI]{bredon72}).
As the action is smooth the exponential map $\exp$ defined (it is defined canonically in terms of the Riemannian metric) on a neighbourhood of the origin in $T_x(X)$ is $G_x$-equivariant.
The $\exp$ map takes a small open disc about the origin homeomorphically onto a neighbourhood of $x$.
If the disc is small enough, then the neighbourhood is a linear slice.
 
Another consequence of the locally smooth action is that the fixed point set of a finite subgroup of $G$ is a \emph{locally flat} submanifold of $X$ embedded as a closed subset. 
Recall that $Y$ is a locally flat submanifold of $X$ if the pair $(X, Y)$ is locally homeomorphic to $(\R^l, \R^k)$. 

\bd{def3s1}An involutive diffeomorphism $r$ on $X$ is a \emph{reflection} if the $r$-action is locally smooth.\ed

In case of $X = \R^l$ one can take a linear isometry as a locally smooth involution with the fixed subset a hyperplane, which separates the ambient space.
However, in general the fixed subset $X_r$ could have components of different dimensions and need not separate the manifold. 
We say that a reflection $r$ is \emph{dissecting} if the fixed set $X_r$ separates $X$.
If the reflection is dissecting then $X_r$ is a codimension $1$ submanifold and its complement has $2$ connected components which are interchanged by $r$ \cite[Lemma 10.1.3]{davisbook08}.
Unless otherwise stated from now on we only consider dissecting reflections.

We note that this assumption covers a fairly large class of examples. 
For instance, let $Y$ be a codimension-$1$ submanifold of $X$ embedded as a closed subset.
Then the mod $2$ homology exact sequence of the pair $(X, X\setminus Y)$ tells us that the complement $X\setminus Y$ has either $1$ or $2$ connected components.
Moreover if $H_1(X, \Z_2) = 0$ then the complement has exactly two connected components.
Hence, in case of simply connected manifolds or homology spheres (of dimension at least $2$) all the reflections with connected codimension-$1$ fixed sets are dissecting.
A general criterion for the separation property is stated in the following lemma.

\bl{lem1sec1} Let $X$ be a connected $l$-manifold and let $Y$ be a connected, $(l-1)$-submanifold embedded in $X$ as a closed subset. Then $X\setminus Y$ has exactly $2$ connected components if and only if the inclusion induced homomorphism $H_c^{n-1}(X, \Z_2)\to H^{n-1}_c(Y, \Z_2)$ (on cohomology with compact support) is trivial. \el

\begin{proof}
The proof is a simple diagram chase:
\[\begin{CD}
  H^{n-1}_c(X, \Z_2) @> >> H^{n-1}_c(Y, \Z_2) \cong \Z_2\\
    @V\cong VV @VV\cong V \\
    H_1(X, \Z_2) @> >> H_1(X, X\setminus Y, \Z_2)@> >> \tilde{H}_0(X\setminus Y, \Z_2) @> >> 0\dispunct{.}\end{CD}
   \] 
The vertical isomorphisms are the duality isomorphisms. 
\end{proof}

%\begin{tikzcd}
% H^{n-1}_c(X, \Z_2) \rar \dar{\cong} &  H^{n-1}_c(Y, \Z_2) \cong \Z_2 \dar{\cong}\\
% H_1(X, \Z_2) \rar & H_1(X, X\setminus Y, \Z_2) \rar &  \tilde{H}_0(X\setminus Y, \Z_2) \rar & 0 
%\end{tikzcd}

%. 

%\bl{lem2sec1}Suppose $r, s$ are distinct reflections on a connected manifold $X$ and that the group $\langle r, s\rangle$ that they generate acts properly and smoothly on $X$. 
%Then 
%\begin{enumerate}
%\item $X_r$ is nowhere dense in $X$.
%\item $X_r\cap X_s$ is nowhere dense in $X_r$.
%\end{enumerate}
%\el
We now explore the structure of the group of diffeomorphisms generated by finitely many reflections.
Let $W$ denote a discrete group of diffeomorphisms that is generated by finitely many reflections and whose action on $X$ is proper and smooth. 
Let $R$ denote the set of all reflections in $W$ (i.e., conjugates of the generating reflections). 
For every $r\in R$ the fixed set $X_r$ is the \emph{wall} associated to $r$ (note that a reflection is uniquely determined by its wall). 
Denote the collection of walls in $X$ by $\A_W := \{X_r\mid r\in R \}$.
Let $\mathcal{C}(\A_W)$ be the set of connected components of the complement $X\setminus \bigcup_{r\in R}X_r$. 
Elements of $\mathcal{C}(\A_W)$ are called \emph{chambers}.
A \emph{wall of a closed chamber} $C$ is a wall $X_r$ such that $C\cap X_r$ is non-empty.
In this case the set $C_r := C\cap X_r$ is a \emph{mirror} of $C$.
Two chambers are \emph{adjacent} if they intersect in a common mirror. 

Fix a chamber $C$ and call it the \emph{fundamental chamber} (note that this is the strict fundamental domain for the $W$ action). 
Let $S$ be the set of reflections $s$ such that $X_s$ is a wall of $C$; these are the \emph{simple reflections}. 
Recall that a \textit{mirror structure} on a space is just collection of closed subspaces indexed by a set \cite[Chapter 5]{davisbook08}.
The family $(C_s)_{s\in S}$ is the \emph{tautological mirror structure} of $C$ and the chamber is called the \emph{mirrored space} over $S$.
For an $x\in C$ we denote by $S(x)$ the set of walls that contain $x$; note that $S(x) = \emptyset$ if $x\in C^{\circ}$.
The following important theorem justifies the term \textit{Coxeter transformation groups} for such groups; see \cite[Theorem 10.1.5]{davisbook08}, \cite[Theorem 3.10]{michoretal07} and \cite[Theorem 1]{gutkin86}. 
This theorem is indeed true for reflection groups acting on topological and generalized manifolds, see \cite[Corollary 1.1]{straume81}.

\bt{thm1sec1} Let $W$ be a discrete group of diffeomorphisms acting properly and smoothly on a connected manifold $X$, let $C$ be a fundamental chamber, let $S$ denote the set of all simple reflections and for $s, r\in S$ let $m_{sr}$ denote the order of $sr$.
Then $W$ is a Coxeter group with the presentation
\[W = \langle S \mid s^2 = 1, ~ (sr)^{m_{sr}} = 1, \forall s, r\in S, r\neq s, 2 \leq m_{sr} \leq \infty \rangle \dispunct{.} \]
\et 

Next we describe the relationship between the fundamental chamber and the ambient manifold. 
The manifold $X$ with $W$-action can be reconstructed from the fundamental chamber $C$ using the \emph{universal construction} of Vinberg \cite{vinberg71} (called the basic construction in \cite[Chapter 5]{davisbook08}).
Define the equivalence relation on $W\times C$ by
\[(g, x)\sim (h, y) \iff x = y,~~ g^{-1}h\in W_x \]
where $W_x$ is the isotropy subgroup. 
Denote the quotient space by $\mathcal{U}(W, C)$ and the elements by $[g, x]$.
The natural $W$-action on the product descends to an action on $\mathcal{U}(W, C)$.
The isotropy subgroup at $[g, x]$ is $gW_{S(x)} g^{-1}$.
The map defined by sending $x\mapsto [1, x]$ is an embedding. 
The projection on the second factor descends to a retraction $p\colon \mathcal{U}(W, C)\to C$.
The construction is universal in the sense that if $D$ is another space with $W$-action and $f\colon C\to D$ is a map such that for all $s\in S, f(C_s)\subset D^s$ (here $D^s$ the fixed set of $s$ on $D$). 
Then there is a unique extension to a $W$-equivariant map $\tilde{f}\colon \mathcal{U}(W, C)\to D$ defined by sending $[g, x]\mapsto g\cdot f(x)$.
\bt{thm2sec1} With the notation as above the natural $W$-equivariant map $\mathcal{U}(W, C)\to X$ induced by the inclusion $C \hookrightarrow X$, is a homeomorphism. \et 

The collection $\A_W$ defines a stratification of the manifold; we now focus on exploring various properties of these strata.
Recall that a homeomorphism between two subsets of $\R^l$ is a diffeomorphism if it extends to a local diffeomorphism on some open neighbourhood. % Does this really need to be recalled?
A second countable, Hausdorff space $C$ is a \emph{smooth $l$-manifold with corners} if it is equipped with a maximal atlas of local charts from open subsets of $C$ to open subsets of the standard $l$-simplicial cone $\R^l_+ := [0, \infty)^l$ so that the transition functions are diffeomorphisms. 
Given a point $x\in C$ and its local coordinates $(x_1,\dots, x_l)\in \R^l_+$ in a neighbourhood the number $c(x) := |\{i\mid x_i = 0 \}|$ is independent of the choice of coordinates and is known as the \emph{depth of $x$}. 
For $0\leq k\leq l$ a connected component of $c^{-1}(k)$ is a \emph{pure stratum} of codimension-$k$ and a \emph{stratum} is the closure of a pure stratum. 
A smooth manifold with corners is \emph{nice} if each stratum of codimension $2$ is contained in the closure of exactly two strata of codimension $1$. 
As a consequence of niceness we have that the closure of a codimension-$k$ stratum is also a smooth manifold with corners. 
Simplicial cones and simple convex polytopes in an euclidean space are examples of smooth, nice manifolds with corners.

\bd{def4s1}
A \emph{mirrored manifold with corners} is a smooth, nice manifold with corners $C$ together with a mirror structure $(C_s)_{s\in S}$ on $C$ indexed by some set $S$ such that
\begin{itemize}
	\item each mirror is a disjoint union of closed codimension one strata of $C$ and 
	\item each closed codimension one stratum is contained in exactly one mirror.
\end{itemize}
\ed
The following theorem describes the structure a closed chamber \cite[Proposition 10.1.9]{davisbook08}.

\bt{thm3sec1}Let $W$ be a Coxeter transformation group acting properly and smoothly on a smooth $l$-manifold $X$. Let $C$ be a fundamental chamber endowed with its tautological mirror structure $(C_s)_{s\in S}$. 
Then $C$ is a mirrored manifold with corners. \et 

There is a converse to the above theorem for which we need one more piece of terminology. The tautological mirror structure of $C$ is said to be \emph{$W$-finite} if for any subset $T\subseteq S$ such that the subgroup generated by $\langle s\in T\rangle$ is infinite then the corresponding intersection of mirrors $\bigcap_{s\in T} C_s$ is empty.

\bt{thm4sec1}
Let $(W, S)$ be a Coxeter group and $C$ is a mirrored manifold with corners with $W$-finite mirror structure $(C_s)_{s\in S}$. Then $\mathcal{U}(W, C)$ is a manifold and $W$ acts properly and smoothly on it as a group generated by reflections.
\et 

\section{Manifold reflection arrangements}
\label{sec:nerve lemma}
In this section we focus on a particularly nice class of manifolds with Coxeter group action.
Our aim is to introduce an analogue of the space $M_W$ and study its homotopy type and that of its quotient by the Coxeter group action.
We achieve our aim by extending the Salvetti complex construction.
We begin by defining the reflection arrangements for manifolds.
\subsection{Definition and examples}
First we isolate important characteristics of the finite reflection arrangements in $\R^l$ that help determine the topology of $M_W$:
\begin{enumerate}
\item the fixed point sets of reflections are codimension-$1$ subspaces that separate $\R^l$;
\item the intersections of these hyperplanes define a stratification of $\R^l$ into open polyhedral cones;
\item the geometric realization of the face poset of this stratification has the homotopy type of $\R^l$.
\end{enumerate}
It is clear that in order to generalize the above properties to the setting of smooth manifolds we need to consider Coxeter transformation groups acting smoothly and properly. 
The walls (i.e., codimension-$1$ submanifolds) fixed by the reflections in $W$ (since they are dissecting) will serve the purpose of reflecting hyperplanes.
 
Let $W$ be such a Coxeter transformation group acting on $X$ and let $\{H_1,H_2,\dots\}$ be the set of all connected components of walls.
We denote by $\L$ the set of all non-empty intersections of $H_i$'s and by $\L^d$ the subset of codimension-$d$ intersections.
For example, we have $\L^0 = \{X\}$ and $\bigcup\L^1 = \bigcup_i H_i$.
For each $d\geq 0$ define the \emph{set of codimension-$d$ strata}
\[\operatorname{St}^d(X) = \text{connected components of } \bigcup\L^d \setminus \bigcup \L^{d+1} \dispunct{.}\]
The set $\operatorname{St}(X) = \bigcup_{d\geq 0} \operatorname{St}^d(X)$ is equipped with the partial order given by the topological inclusion. 
Note that $X$ is a disjoint union of the subsets in $\operatorname{St}(X)$.
The reader can check that this defines a \emph{stratification} in the sense of \cite[Definition 2.1]{tamaki01}.

We would like the stratification of $X$ induced by the reflecting submanifolds to satisfy condition (2) and the corresponding face poset to satisfy condition (3) above.
%This is not true in general.
%For example, consider the action of $\Z/2$ on $S^2$ which fixes an equator and permutes the two hemispheres; none of the strata in this case resemble polyhedra and the resulting face poset does not have the homotopy type 
From the previous section we know that each closed stratum is a nice manifold with corners, however, the resulting face poset need not realize the homotopy type of $X$.
We propose the following definition.
\bd{def1sec3}
Let $X$ be a smooth $l$-manifold, let $W$ be a Coxeter transformation group acting properly and smoothly and let $R$ be the set of reflections in $W$.
The \emph{(manifold) reflection arrangement} corresponding to $W$ is the finite collection 
\[ \A_W = \{X_r\mid r\in R\}\] 
of walls fixed by $W$ given that the following conditions are satisfied:
\begin{enumerate}
\item The stratification of $X$ induced by the intersections of these submanifolds define the structure of a regular cell complex.
\item Each closed chamber, as a nice manifold with corners, is combinatorially equivalent to a simple, convex polytope.
\end{enumerate}
\ed
The face poset of a nice manifold with corners is the set of all connected components of $c^{-1}(k)$ for all $k$ and ordered by topological inclusion. Two nice manifolds with corners are said to be \textit{combinatorially equivalent} if their face posets are isomorphic.
In view of \cite[Corollary 5.2]{wiemeler13}, if $P$ is the simple convex polytope combinatorially equivalent to the (closed) fundamental chamber $\ol{C}$ then there is a diffeomorphism of nice manifolds with corners $f\colon \ol{C}\to P$ that induces the given isomorphism on the face posets.
In the language of \cite{davis14}, we assume that the closed chambers are Coxeter orbifolds of type (III).

We denote the face poset of a manifold reflection arrangement by $\mathcal{F}(\A_W)$ (we will drop the subscript $W$ if the context is clear).
According to \Cref{thm3sec1} this cellular decomposition is simple, i.e., a codimension-$k$ cell is in the closure of $k$ codimension-$1$ cells.
As before the codimension-$0$ faces are called \emph{chambers}.
The set of chambers of the arrangement $\A_W$ will be denoted by $\mathcal{C}(\A_W)$.
From now on $C$ will always denote an open chamber.

\begin{remark}
By \Cref{thm1sec1}, the groups $W$ that appear in this context are Coxeter groups, and hence appear as reflection groups on some $\R^n$.
The obvious examples of manifold reflection arrangements are obtained by restricting a reflection group acting on euclidean space to an invariant embedded submanifold, see \Cref{ex1,ex2}.
For less straightforward examples one should look at Davis' seminal work on exotic aspherical manifolds. 
In dimension $l > 4$ he considers a compact contractible manifold whose boundary is a homology sphere (preferably not simply connected). 
Then he constructs an open contractible $l$-manifold with a (right-angled) Coxeter group action such that the fundamental domain is the given compact contractible manifold and its mirrors form the homology sphere.
Further details of this construction and the groups are beyond the scope of this paper, hence we refer the interested reader to \cite[Chapter 10]{davisbook08}.
\end{remark}

\be{ex1}
Consider the dihedral group of order $2m$ generated by two generators, say $\{s, t\}$. 
Its natural reflection action on $\R^2$ fixes a union of $2m$ lines passing through the origin. 
The intersection of this line arrangement with the unit circle $S^1$ gives us a reflection arrangement on $S^1$.
The fundamental chamber is a $1$-simplex and its mirrors are two points labeled $C_s$ and $C_t$.
Note that the two element set $\{s, t\}$ is spherical but $C_s\cap C_t = \emptyset$.
In particular, consider the action of $S_3$, the symmetric group on $3$ letter on $S^1$ (\Cref{s3ons1}). 
The chambers (the $6$ arcs) are labeled by the group elements whereas the fixed submanifolds (the $6$ vertices) are labeled by the one-generator parabolic subgroups.
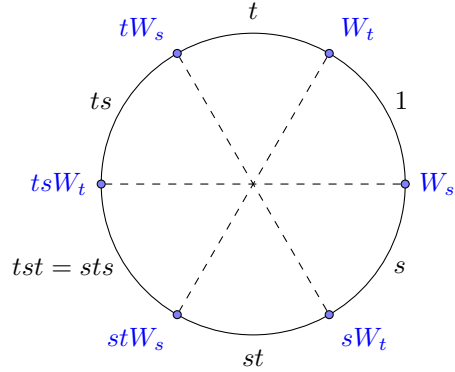
\begin{figure}[htb]
    \centering
\begin{tikzpicture}[fix/.style={circle,draw,fill=blue!50,thin,
inner sep=0pt,minimum size=3pt}]
\draw (0,0) circle (2);
\node (Ws) at (2,0) [label={[blue]right:$W_s$},fix] {};
\node (Wt) at (1,1.732) [label={[blue]60:$W_t$},fix] {};
\node (tWs) at (-1,1.732) [label={[blue]120:$tW_s$},fix] {};
\node (tsWt) at (-2,0) [label={[blue]left:$tsW_t$},fix] {};
\node (stWs) at (-1,-1.732) [label={[blue]240:$stW_s$},fix] {};
\node (sWt) at (1,-1.732) [label={[blue]300:$sW_t$},fix] {};

\draw[dashed] (Wt) -- (stWs);
\draw[dashed] (Ws) -- (tsWt);
\draw[dashed] (sWt) -- (tWs);

\node at (1.95,1.1) {$1$};
\node at (0,2.3) {$t$};
\node at (-2,1.1) {$ts$};
\node at (1.93,-1.1) {$s$};
\node at (0,-2.3) {$st$};
\node at (-2.5,-1.1) {$tst = sts$};

%\draw[fill=red] (0) circle (1pt);

\end{tikzpicture}

  \caption{The $S_3$ action on $S^1$.}     \label{s3ons1}     
\end{figure}

In general, any reflection group acting on $\R^{l+1}$ via linear isometries restricts to an action on $S^l$.
This provides us with a number of examples of manifold reflection arrangements.
\ee

\be{ex2}
Let $C$ be the $1$-simplex and let $s$ and $t$ be the reflections across its endpoints and set $m_{st} = \infty$.
The reader can verify that the group generated by these two elements is the infinite dihedral group $\Z/2 \ast \Z/2$ and the manifold $\mathcal{U}(W,C)$ is $\R^1$.
\ee

\be{ex3}
If the fundamental chamber $C$ is combinatorially equivalent to the unit simplex of dimension at least $2$ and $W$ is a finite Coxeter group of rank at least $3$ then the manifold $\mathcal{U}$ is homeomorphic to the unit sphere (in fact, $W$ can be made to act via isometries \cite[Lecture 3]{davis11}).
In this case the manifold reflection arrangement is a collection of codimension-$1$ sub-spheres. 
All the positive-dimensional intersections are connected spheres.
\ee

\be{ex4}
Let $\mathbb{X}^l$ denote either $\R^l$ or $S^l$ or the hyperbolic space $\mathbb{H}^l$. 
Let $K$ be a (simple) convex polytope in $\mathbb{X}^l$ with nonobtuse dihedral angles (see \cite[Section 6.3]{davisbook08} for precise definition).
Denote by $S$ the set of reflections across the codimension-$1$ faces of $K$.
The collection $\{ K_s\mid s\in S\}$ of codimension-$1$ faces defines a mirror structure on $K$.
If $K_s\cap K_t\neq\emptyset$ then it is a codimension-$2$ face with the dihedral angle $\frac{\pi}{m_{st}}$, for some integer $m_{st}\geq 2$.
If such an intersection is empty then put $m_{st} = \infty$.
It follows from \cite[Theorem 6.4.3]{davisbook08} that the group $W \subsetneq \mathrm{Isom}(\mathbb{X}^l)$ generated by $S$ is a Coxeter group.
The universal space $\mathcal{U}(W, K)$ is $W$-equivariantly homeomorphic to $\mathbb{X}^l$.
The $W$-action is proper and $K$ is a strict fundamental domain.
\ee

%As a particular example, consider the action of $S_3$ the dihedral group of order $6$ on $\R^2$.
%The intersection of the fixed lines with the unit circle gives us the reflection arrangement of $0$-spheres in $S^1$.
%The tangent bundle complement in this case is the cylinder minus $6$ points; it has the homotopy type of wedge of $7$ circles.
%On the other hand the complexified complement of the union of $6$ lines fixed by the dihedral group in $\C^2$ is $2$-dimensional and its fundamental group is not a free group.
 
%If $\ol{C}$ is combinatorially equivalent to a $k$-gon and $W$ is the corresponding hyperbolic polygonal group then the manifold $\mathcal{U}$ is the hyperbolic plane (see \cite[Chapter 6]{davisbook08}).
%The space $M_W$ is the complement inside the tangent bundle of $\mathbb{H}^2$.

We say that a wall $X_r$ of $\A$ \emph{separates} two chambers $C$ and $D$ if they are contained in distinct connected components of $X\setminus X_r$. 
For two chambers $C$ and $D$, the set of all walls that separate these two chambers is denoted by $\RR(C, D)$. 
The \emph{distance between two chambers} is the cardinality of the set $\RR(C, D)$ and is denoted by $d(C, D)$. 
\bp{lem1sec3}
Let $X$ be an $l$-manifold and $\A$ be an arrangement of submanifolds, let $C_1, C_2, C_3$ be three chambers of this arrangement. 
Then, 
\[\RR(C_1, C_3) = [\RR(C_1, C_2)\setminus \RR(C_2, C_3)] \cup [\RR(C_2, C_3)\setminus \RR(C_2, C_1)]. \] 
\ep

\bd{def2sec3}
Let $\A_W$ be a manifold reflection arrangement in $X$, corresponding to $W$.
For a point $x\in X$ the \emph{local arrangement} at $x$ is 
\[\A_x \vcentcolon= \{X_r\in \A_W\mid x\in X_r \}. \]
For a face $F\in \mathcal{F}(\A_W)$ the local arrangement at $F$ is
\[ \A_F \vcentcolon= \{X_r\in \A_W\mid F\subseteq X_r \}.\]
%Let $\A^{'}_F$ denote the set of connected components of $\A_F$. The restriction of a local arrangement to an open set $V\subseteq X$ is
%\[\A_F|_V := \{H\cap V\mid H\in \A^{'}_F \}. \]
%\todo{Is this needed ?}
\ed

A local arrangement $\A_F$ need not be a reflection arrangement in the sense of \Cref{def1sec3}.
However it does define a stratification of $X$ which we denote by the pair $(X, \A_F)$. 
A \emph{map of stratified spaces} is a continuous map that induces an order preserving map on the corresponding face posets (such maps are called strict morphisms in \cite[Definition 2.15]{tamaki01}).
Define a map from $(X, \A)$ to $(X, \A_F)$ by sending a face $G$ to the face of $\A_F$ of least dimension that contains $G$.
The reader can verify that this is a map of stratified spaces.
We denote by $\pi_F$ the induced map on the face posets.  
We also use the notation $C_F$ for the chamber $\pi_F(C)$ of $\A_F$. 
 
%Define an action of a face $F$ on a chamber $C$ as follows: 

\bd{def3sec3}
Given a face $F$ and a chamber $C$ denote by $F\circ C$ a chamber that satisfies:
\begin{enumerate}
	\item $F \subseteq \overline{F\circ C}$,
	\item $\pi_F(C) = \pi_F(F\circ C)$,
	\item $d(C, F\circ C) = \hbox{~min~} \{d(C, C') \mid C'\in \Ch, F\subseteq\overline{C'} \}$. 
\end{enumerate}
\ed

The following result is evident.
\bp{lem2sec3}
The chamber $F\circ C$ always exists and is unique.
\ep

The reader can verify that if $F\leq C$ then $F\circ C = C$ and for $F\leq F'$ we have $F'\circ (F\circ C) = F'\circ C$.

Before moving on we should convince the reader that there are examples other than the classical ones.
Since each chamber $C$ is simply connected and for every $s\in S$ the closed codimension-$1$ face $C_s$ is non-empty the manifold $X \cong \mathcal{U}(W, C)$ is also simply connected if and only if for each spherical subset $\{s, t\}$ we have $C_s\cap C_t\neq \emptyset$ (\cite[Theorem 9.1.3]{davisbook08}). Moreover the manifold $X$ is contractible if and only if for every non-empty spherical subset $T$ the intersection $\cap_{s\in T} C_s$ is acyclic (\cite[Theorem 9.1.4]{davisbook08}).

\subsection{The tangent bundle complement} \label{subsec:tbc}

In the classical case the action of the reflection group on $\R^l$ extends to $\C^l$ by extension of scalars. 
Note that topologically $\C^l$ is the tangent bundle of $\R^l$. 
The Coxeter transformation group acts locally smoothly on the manifold and hence its action naturally extends to the tangent bundle. 
The fixed points of this action are the tangent bundles of the reflecting submanifolds. 
Hence the complement of the union of these tangent bundles over reflecting submanifolds is topologically a generalization of the complexified complement.
\bd{def4sec3}
The \emph{tangent bundle complement} associated with the manifold reflection arrangement $\A_W$ in $X$ is denoted $M(\A_W)$ or $M_W$ and defined as 
\[M_W \vcentcolon= TX \setminus \bigcup_{r\in R} TX_r. \]
\ed
Note that the Coxeter group action on $M_W$ is a covering space action. 
Denote the orbit space of this action by $N_W$. 
It is clear that the group $\pi_1(M_W)$ is an analogue of the pure Artin group whereas $\pi_1(N_W)$ is an analogue of the Artin group in the sense that they fit in an analogous exact sequence.

%The tangent space at $x$, $T_xF$ of a manifold with corner $F$ has an intrinsically defined convex cone comprising of the `inner' tangent vectors, this is called the inner sector (see \cite{joyce12}).
%The local linearity assumption tells us that the chambering of some neighborhood of a point $x$ of $X$, under taking inner sectors, gives us the chambering of $T_xX$. \todo{Where should this remark go?}

%In this subsection let $X$ be a smooth $l$-manifold, $W$ a Coxeter transformation group and $\A$ the corresponding manifold reflection arrangement.

\subsection{The Salvetti complex construction} \label{subsec:sal}

The aim is to prove that there is exists a regular cell complex that is $W$-equivariantly homotopy equivalent to the tangent bundle complement $M_W$.

\begin{definition}
The \emph{Salvetti poset} of a manifold reflection arrangement $\A_W$ in an $l$-manifold $X$ is the set 
\[\Sal_0(\A_W) = \set{(F,C) \in \F(\A_W) \times \mathcal{C}(\A_W)}{F \le C}\]
with the following relation $\preceq$ on $\Sal_0(\A_W)$ as a partial order
\[(F,C) \preceq (G,D) \text{ if } F \le G \text{ and } C_F \subseteq D_G \dispunct{.}\]
The \emph{Salvetti complex} of $\A_W$, denoted by $\Sal(\A_W)$, is defined to be the geometric realization of the derived complex of $(\Sal_0(\A_W),\preceq)$.
\end{definition}

An alternative way of defining the Salvetti complex, in terms of parabolic subgroups is discussed and shown to be equivalent to this construction in \Cref{subsec:group salvetti}.

\begin{theorem}\label{main theorem}
There exists a $W$-equivariant homotopy equivalence $f \from \Sal(\A_W)\to M_W$.
This induces a homotopy equivalence $\bar f \from \Sal (\A_W)/W \to M_W/W$.
\end{theorem}

The proof closely imitates that of Theorem~3.1 in \cite{parisarxiv}.
Since our arrangements are assumed to originate from a group action, we deal with only the equivariant case of the constructions.
The principal step is to find a nerve of $M_W$ that equivariantly corresponds to $\Sal_0(\A)$, so that we can apply the equivariant nerve lemma \cite[Proposition~2.2]{parisarxiv}.

In the proof in \cite{parisarxiv}, the linear structure of $\R^k$ is heavily used.
In particular, since the tangent bundle is globally trivial, the open sets $U(F,C)$ can be defined as products.
Since we instead are working on a manifold, we need to pullback similar product constructions to the tangent bundle.
The convex structure on faces, defined below, lets us establish analogous combinatorial properties for our construction.

\begin{definition} If $F \in \FA$ is a face of dimension $k$, define a \emph{convex structure} on $F$ to be a diffeomorphism $\varphi \from P \to \ol{F}$ for some simple convex polytope $P \subset \R^k$.\end{definition} 

Since both the spaces are manifolds with corners, any such diffeomorphism must also be a combinatorial equivalence.
That is to say, faces of $P$ are mapped to faces $F' \le F$, and vice-versa, and $F'' \le F'$ iff $\varphi^{-1}(F'') \le \varphi^{-1}(F')$ for faces $F',F'' < F$.

%\begin{definition}
%Two convex structures $\varphi_F \from P \to \ol{F}$ and $\varphi'_F \from P' \to \ol{F}$ on $F$ are equivalent if $\varphi_F = \varphi'_F \circ T$ for some invertible affine transformation $T \from \R^k \to \R^k$.
%\end{definition}
%
%Since the invertible affine transformations $Tx = Ax+b$ for $A \in GL_k(\R)$ and $b \in \R^k$ form a group, the equivalence of convex structures is an equivalence relation.

%\begin{lemma}\label{thm:preserve convexity}
%Suppose $\varphi$, $\varphi'$ are equivalent convex structures on $F \in \FA$, and $x_0,x_1,\dots,x_p$ are points of $\ol{F}$.
%Then for any $(t_0,t_1\dots,t_p) \in \Delta^{p}$, 
%\[\varphi\Big(\sum_{i=0}^p t_i \varphi^{-1}(x_i)\Big)= \varphi'\Big( \sum_{i=0}^p t_i \varphi'^{-1}(x_i)\Big) \dispunct{.}\]
%\end{lemma}
%
%\begin{proof}
%Let $\varphi(y_i) = x_i$ and $\varphi(z_i) = x_i$ for $i=0,1,\dots,p$.
%Then for some affine transformation $T \from \R^k \to \R^k$, $z_i = Ty_i$ for each $i$.
%Let $Ty = Ay + b$, where $A$ is an invertible matrix and $b \in \R^k$.
%Then
%\[z_0 = Ay_0 + b = A\sum_{i=1}^p t_i y_i + b = \sum_{i=1}^p t_i (Ay_i + b) = \sum_{i=1}^p t_i z_i\]
%as required.
%\end{proof}

\begin{lemma}
There is a collection of convex structures $\varphi_F$ for each $F \in \FA$ such that for any $F < F'$, if $P = \varphi_{F'}^{-1}(\ol{F})$, then $\varphi_{F'}|_{P} = \varphi_F$.
Further, these can be chosen in an equivariant way, that is, such that $\varphi_{wF} = w \circ \varphi_F$ for $w\in W$.
\end{lemma}

\begin{proof}
Fix a fundamental chamber $C_0 \in \mathcal{C}(\A)$.
By the remark after \Cref{def1sec3}, we can give $C_0$ a convex structure $\varphi = \varphi_{C_0}$.
Now, for the chamber $wC_0$ for some $w \in W$, define the convex structure by $\varphi_{wC_0} = w \circ \varphi$.

For a face $F$ of $C_0$, restrict $\varphi$ to $\varphi^{-1}(\ol{F})$, that is, set $\varphi_F = \varphi|_{\varphi^{-1}(\ol{F})}$, and similarly for a face $F$ of $wC_0$.
It is enough to check well-definedness on faces $F$ of $C_0$. 
But if $F$ is a face of both $C_0$ and $wC_0$, then $w$ must fix $\ol{F}$ point-wise (since $F$ is fixed by a unique standard parabolic subgroup, see \Cref{lem2sec4}).
Hence, on $\varphi^{-1}(\ol{F})$, $w \circ \varphi = \varphi$.
\end{proof}

So fix a collection of convex structures $\varphi_F$ as above.
Let $\mathring\Delta^n$ denote the \emph{open} unit simplex in $\R^{n+1}$ and $\Delta^n$ the closed unit simplex.
If $x_0,x_1,\dots,x_n$ are points contained in $\ol{F}$ for some face $F$, and $(t_0,t_1,\dots t_n) \in \Delta^n$, then the point $\varphi_F\mleft(\sum t_i \varphi_F^{-1}(x_i)\mright)$ does not depend on the choice of $F$.
So we shall abuse notation and denote this by $\sum t_i x_i$, as long as $x_i$ are contained in the closure of some face $F$.
We then also have, for any $w \in W$, $w(\sum t_i x_i) = \sum t_i w(x_i)$.

Before we define the open cover $\{ U(F,C)\mid F\leq C \}$ of $M(\A)$, indexed by $(F,C) \in \Sal_0(\A)$, we will first cover $X$ by open sets $\omega(F)$ corresponding to $F \in \FA$.
Then $U(F,C)$ will be an open subset of $T\omega(F)$.

Throughout the rest of this section we use the following definitions and notations.
A \emph{chain of length $p+1$} in $\FA$ is a sequence $(F_0,\dots,F_p)$ in $\FA$ such that $F_0 < F_1 <\dots < F_p $.
We define two partial orders on chains.
Let $\gamma = (F_0,\dots,F_p)$ and $\gamma' = (F'_0,\dots,F'_q)$. 
First, $\gamma \le \gamma'$ if $F_0 = F'_0$ and $\{F_0,\dots,F_p\} \subseteq \{F'_0,\dots,F'_q\}$.
Second, $\gamma \preceq \gamma'$ if $p \le q$, and for $0 \le i \le p$, $F_{p-i} = F'_{q-i}$.
Note that if $F_0 = F_0'$ (in particular if $\gamma \le \gamma'$) and $\gamma \preceq \gamma'$ then $\gamma = \gamma'$.

For $F \in \FA$, denote by $\Chain(F)$ the set of chains $(F_0,\dots,F_p)$ such that $F = F_0$.
More generally, let $\Chain(\gamma)$ denote the set of chains $\gamma'$ such that $\gamma \le \gamma'$.
For a chain $\gamma = (F_0,F_1,\dots,F_p)$, and $w \in W$, let $w\gamma$ be the chain $(wF_0,wF_1,\dots,wF_p)$. 

For each $F \in \FA$, fix a point $x(F) \in F$ so that $wx(F) = x(wF)$ for all $w \in W$. For a given chain $\gamma = (F_0,F_1,\dots,F_p)$, let $x_i = x(F_i)$ and define the following subset of $F_p$:
\[\Theta(\gamma) = \set{t_0z_0 + t_1x_1 + \dots + t_px_p}{ z_0 \in F_0, (t_0,t_1,\dots,t_p) \in \mathring{\Delta}^p}\dispunct{.}\]

\begin{lemma}\label{lemma:inductive theta}
Let $\gamma = (F_0,F_1,\dots,F_p)$ be a chain for some $p>0$ and let $\gamma_1 = (F_0,F_1,\dots,F_{p-1})$.
Then
\[\Theta(\gamma) = \set{t z + (1-t)x(F_p)}{z \in \Theta(\gamma_1), 0 < t < 1} \dispunct{.}\]
\end{lemma}
\begin{proof}
Let $x_i = x(F_i)$.
Using that $\mathring\Delta^p = \set{(tx,1-t)}{x \in \mathring\Delta^{p-1},0<t<1}$, it is enough to note that by definition,
\[\Theta(\gamma_1) = \set{\mleft(t_0z_0 + t_1x_1 + \dots + t_{p-1}x_{p-1}\mright)}{ z_0 \in F_0, (t_0,t_1,\dots,t_{p-1}) \in \mathring\Delta^{p-1}} \dispunct{.} \qedhere\]
\end{proof}

\begin{lemma}\label{lemma:theta equivariant}
For any $w \in W$, and any chain $\gamma$,
\[w\Theta(\gamma) = \Theta(w\gamma)\dispunct{.}\]
\end{lemma}
\begin{proof}
Let $x_i = x(F_i)$ and $u \in \Theta(\gamma)$. 
By definition, $u = t_0z + \sum_{i=1}^p t_ix_i$ for some $z \in F_0$, and some $(t_0,t_1,\dots,t_p) \in \mathring\Delta^p$.
But then if $x_i' = x(wF_i) = w(x_i)$, we must have that $wu = t_0z + \sum_{i=1}^p t_i x_i'$, and so $wu \in \Theta(w\gamma)$.

Thus, $w\Theta(\gamma) \subseteq \Theta(w\gamma)$.
Similarly, $w^{-1}\Theta(w\gamma) \subseteq \Theta(\gamma)$, and so we are done.
\end{proof}

\begin{lemma}\label{lemma:thetas intersect then preceq}
Let $\gamma = (F_0,F_1,\dots,F_p)$ and $\gamma' = (F'_0,F'_1,\dots,F'_q)$ be two chains.
If $p \le q$ and $\Theta(\gamma) \cap \Theta(\gamma') \ne \emptyset$, then $\gamma \preceq \gamma'$.
\end{lemma}
\begin{proof}
Let $u \in \Theta(\gamma) \cap \Theta(\gamma')$.
Then $u \in F_p$ and $u \in F'_q$, and hence $F_p = F'_q$.
Now we proceed by induction on $p$.
If $p = 0$, we are already done, so assume $p \ge 1$ and the induction hypothesis.
Let $\gamma_1 = (F_0,F_1,\dots,F_{p-1})$ and $\gamma'_1 = (F'_0,F'_1,\dots,F'_{q-1})$.

Let $x = x(F_p)$.
Then $u = tz+(1-t)x = t'z' + (1-t')x$ for some $0<t,t' \le 1$, $z \in \Theta(\gamma_1)$, and $z' \in \Theta(\gamma'_1)$.
Thus, $z,z'$ are points in $\partial F_p$ and are on the ray $xu$.
Since $x$ is an interior point of $\ol{F_p}$, this implies that $z = z'$, and hence $\Theta(\gamma_1) \cap \Theta(\gamma'_1) \ne \emptyset$.
By the induction hypothesis, $\gamma_1 \preceq \gamma'_1$ and hence $\gamma \preceq \gamma'$.
\end{proof}

\begin{lemma}\label{lemma:preceq then theta subset}
Let $\gamma = (F_0,F_1,\dots,F_p)$ and $\gamma' = (F'_0,F'_1,\dots,F'_q)$ be two chains.
If $\gamma \preceq \gamma'$, then $\Theta(\gamma') \subseteq \Theta(\gamma)$.
\end{lemma}
\begin{proof}
If $p = q$, there is nothing to prove, so assume $p<q$.
For $i =1,2,\dots,q$, let $x_i = x(F'_i)$.

Now let $u \in \Theta(\gamma')$.
Then for some $(t_0,t_1,\dots,t_q) \in \mathring\Delta^q$, and $z \in F_0'$,
\[u = t_0 z + \sum\nolimits_{i=1}^{q} t_ix_i \dispunct{.}\]
Then let $t'_0 = \sum_{i=0}^{q-p}t_i > 0$ and 
\[u' = \frac{1}{t'_0}\mleft[t_0 z + \sum\nolimits_{i=1}^{q-p} t_ix_i\mright] \dispunct{,}\]
which is a point in $\ol{F'_{q-p}} = \ol{F_0}$.
Let $t'_i = t_{q-p+i}$ and note that $x(F_{i}) = x_{q-p+i}$.
But now $(t'_0,t_{q-p+1},\dots,t_q)$ is a point in $\mathring\Delta^p$ and so $u = t_0'u' + \sum_{i=1}^{p}t_i' x(F_i) \in \Theta(\gamma)$.
\end{proof}

For a given $F \in \FA$, set 
\[\omega(F) = \bigcup_{\mathclap{\gamma \in \Chain(F)}} \Theta(\gamma)\dispunct{.}\]
More generally, for a chain $\gamma$,
\[\omega(\gamma) = \bigcup_{\mathclap{\gamma' \in \Chain(\gamma)}} \Theta(\gamma') \dispunct{.}\]

\begin{lemma}\label{lemma:inductive omega}
Let $\gamma = (F_0,F_1,\dots,F_p)$ be a chain and $G \in \FA$ be any face with $\dim G > \dim F_p$. Then
\[\ol{G} \cap \omega(\gamma) = \set{t z + (1-t) x(G)}{ z \in \partial G \cap \omega(\gamma), 0 < t \le 1} \dispunct{.}\]
\end{lemma}
\begin{proof}
Suppose $u \in \ol{G} \cap \omega(\gamma)$.
If $u \in \partial G$, take $z = u$ and $t = 1$.
Otherwise $u \in \Theta(\delta)$ for some $\delta = (G_0,G_1,\dots,G_q) \in \Chain(\gamma)$ such that $G = G_q$.
Since $F_0 = G_0$, and $G$ is of higher dimension than $F_0$, $q>0$.
Define $\delta_1 = (G_0,G_1,\dots,G_{q-1})$.
Then by \Cref{lemma:inductive theta}, for some $z \in \Theta(\delta_1)$ and $0 < t < 1$, $u = tz + (1-t) y$.
But since $G \notin \gamma$, $\delta_1 \in \Chain(\gamma)$.
Thus $z \in \Theta(\delta_1) \subseteq \omega(\gamma)$, and so the forward inclusion is proved.

For the reverse inclusion, let $z \in \partial G \cap \omega(\gamma)$.
Clearly $z \in \omega(\gamma)$, so let $0 < t < 1$.
Now let $\delta_1 = (G_0,G_1,\dots,G_q) \in \Chain(\gamma)$ be such that $z \in \Theta(\delta_1)$.
Then $z \in G_q$, and hence $G_q < G$.
So $\delta = (G_0,G_1,\dots,G_q,G)$ is a chain and since $\gamma \le \delta_1 \le \delta$, $\delta \in \Chain(\gamma)$.
Now again by \Cref{lemma:inductive theta}, $tz + (1-t)y \in \Theta(\delta) \subseteq \ol{G} \cap \omega(\gamma)$.
\end{proof}

\begin{lemma}\label{lemma:omega is open}
Let $F \in \FA$. Then $\omega(F)$ is open in $X$.
\end{lemma}
\begin{proof}
Let $d = \dim F$ and for $k \ge 0$, let $X^{d+k}$ be the $(d+k)$-skeleton.
We prove by induction on $k$ that $X^{d+k} \cap \omega(F)$ is open in $X^{d+k}$.
The set $X^{d} \cap \omega(F) = F$ is of course open in $X^d$, so assume $k \ge 1$ and the induction hypothesis.
Consider a face $G$ of dimension $d+k$ and let $x = x(G)$.

If $F \not < G$, then $\ol{G} \cap \omega(F) = \emptyset$.
Otherwise, by the induction hypothesis, $B = \omega(F) \cap \partial G$ is open in $\partial G$.
Then $\ol{G} \cap \omega(F) = \set{tz+(1-t)x}{z \in B, 0 < t \le 1}$ is open in $\ol{G}$.
So $\omega(F) \cap X^{d+k}$ is open in $X^{d+k}$.
\end{proof}

\begin{lemma}\label{lemma:omega inclusion reversing}
Let $\gamma = (F_0,F_1,\dots,F_p)$ and $\gamma' = (F'_0,F'_1,\dots,F'_q)$ be two chains.
If $\{F_0,F_1,\dots,F_p\} \subseteq \{F'_0,F'_1,\dots,F'_q\}$, then $\omega(\gamma') \subseteq \omega(\gamma)$.
\end{lemma}
\begin{proof}
By \Cref{lemma:preceq then theta subset}, it is enough to show that for any $\delta' \in \Chain(\gamma')$, there is some $\delta \in \Chain(\gamma)$ such that $\delta \preceq \delta'$.
Let $\delta' = (G_0,G_1,\dots,G_r)$.
Since $F_0 \in \{F'_0,F'_1,\dots,F'_q\} \subseteq \{G_0,G_1,\dots,G_r\}$, let $F_0 = G_s$ for some $s \le r$.
Then $\delta = (G_s,G_{s+1},\dots,G_r)$ works.
\end{proof}

\begin{lemma}\label{lemma:thetas unique in omega}
If $u \in \omega(F)$ then there is a unique chain $\gamma \in \Chain(F)$ such that $u \in \Theta(\gamma)$.
\end{lemma}
\begin{proof}
Let $\gamma,\gamma' \in \Chain(F)$ such that $u \in \Theta(\gamma) \cap \Theta(\gamma')$.
Then by \Cref{lemma:thetas intersect then preceq}, without loss of generality, $\gamma \preceq \gamma'$.
But since both are in $\Chain(F)$, we must have $\gamma = \gamma'$.
\end{proof}

\begin{lemma}\label{lemma:omegas intersect then preceq}
Let $F,G \in \FA$.
If $\omega(F) \cap \omega(G) \ne \emptyset$, then $F \le G$ or $G \le F$.
\end{lemma}
\begin{proof}
Let $u \in \omega(F) \cap \omega(G)$.
Then for some $\gamma = (F_0,F_1,\dots,F_p) \in \Chain(F)$ and $\gamma' = (F'_0,F'_1,\dots,F'_q) \in \Chain(G)$, $u \in \Theta(\gamma) \cap \Theta(\gamma')$.
In particular, $F_0 = F$ and $F'_0 = G$.
But by \Cref{lemma:thetas intersect then preceq}, either $\gamma \preceq \gamma'$ or $\gamma' \preceq \gamma$.
Thus, either $F \le G$ or $G \le F$.
\end{proof}

By induction, we get:
\begin{lemma}\label{lemma:omegas intersect then chain}
Let $F_0,F_1,\dots,F_p \in \FA$ be such that $\omega(F_0) \cap \omega(F_1)\cap \dots \cap \omega(F_p) \ne \emptyset$. Then up to a permutation of indices, $F_0 \le F_1 \le \dots \le F_p$.
\end{lemma}

\begin{lemma}\label{lemma:intersect of omegas is omega}
Let $\gamma = (F_0,F_1,\dots,F_p)$ be a chain. Then
\[\omega(F_0) \cap \omega(F_1)\cap \dots \cap \omega(F_p) = \omega(\gamma) \dispunct{.}\]
\end{lemma}
\begin{proof}
By \Cref{lemma:omega inclusion reversing}, $\omega(\gamma) \subseteq \omega(F_0) \cap \omega(F_1)\cap \dots \cap \omega(F_p)$.
Conversely, let $u \in \omega(F_0) \cap \omega(F_1)\cap \dots \cap \omega(F_p)$.
Let $\delta_i \in \Chain(F_i)$ be such that, $u \in \Theta(\delta_i)$.
We want $\delta_0 \in \Chain(\gamma)$ and hence $u \in \omega(\gamma)$. 
It is enough to show that $F_i \in \delta_0$ for each $i=1,\dots,p$.
Now by \Cref{lemma:thetas intersect then preceq}, either $\delta_0 \preceq \delta_i$ or $\delta_i \preceq \delta_0$.
Since $F_0 < F_i$, the former cannot hold, and hence $F_i \in \delta_0$.
\end{proof}

\begin{lemma}\label{lemma:skeleton of omega}
Let $\gamma = (F_0,F_1,\dots,F_p)$ be a chain and let $F_p$ have dimension $d$.
Then $\omega(\gamma) \cap X^d \subseteq F_p$, where $X^d$ is the $d$-skeleton of $X$.
\end{lemma}
\begin{proof}
Consider $G \in \FA$ of dimension at most $d$, and let $u \in G \cap \omega(\gamma)$. 
Then there is some $\delta = (G_0,G_1,\dots,G_q) \in \Chain(\gamma)$ such that $u \in \Theta(\delta)$.
But then $G_q = G$ and since $F_p \subseteq \{G_0,G_1,\dots, G_q\}$, we must have $F_p = G_r$ for some $r \le q$.
In particular, $F_p \le G$.
But then the dimension assumption forces $F_p = G$, so $u \in F_p$.
\end{proof}

\begin{lemma}\label{lemma:omega restricts to contractible}
Let $\gamma = (F_0,F_1,\dots,F_p)$ be a chain. Then $\ol{F_p} \cap \omega(\gamma)$ is contractible.
\end{lemma}

\begin{proof}
By \Cref{lemma:skeleton of omega}, $\ol{F_p} \cap \omega(\gamma) = F_p \cap \omega(\gamma)$.
Let $x_i = x(F_i)$ and 
\[x = \frac1{p+1}(x_0+x_1+\dots+x_p) \in \Theta(\gamma)\dispunct{.}\]
Let $u \in \Theta(\delta) \cap F_p$ where $\delta = (G_0,G_1,\dots,G_q) \in \Chain(\gamma)$.
Then $F_0 = G_0$ and $F_p = G_q$. 
Let $u = \sum_{j=0}^q t_jz_j$ where $z_0 \in F_0 = G_0$, and $z_j = x(G_j)$ for $0<j \le q$.
Since $\delta \in \Chain(\gamma)$, each $y_i$ for $i>0$ occurs as some $z_j$.
Let $t \in (0,1)$.
Then $z_0' = tz_0 + (1-t)x_0 \in G_0$ and so $tx+(1-t)u$ is also a positive convex combination of $z_0'$ and $z_j$ for $j>0$, and hence $tx+(1-t)u \in \Theta(\delta) \subseteq F_p \cap \omega(\gamma)$.
So the segment $ux$ is contained in $\omega(\gamma)$.
This corresponds to star convexity under the convex structure of $F_p$ and hence the contractibility follows.
\end{proof}

\begin{lemma}\label{lemma:omega contractible}
Let $\gamma = (F_0,F_1,\dots,F_p)$ be a chain and let $F_p$ have dimension $d$.
Then $\omega(\gamma)$ deformation retracts to $w(\gamma) \cap X^d$, and hence is contractible.
\end{lemma}
\begin{proof}
By induction, it is enough to show that $\omega(\gamma) \cap X^{n-k+1}$ deformation retracts to $\omega(\gamma) \cap X^{n-k}$ for $1 \le k \le n-d$.

Let $G \in \FA$ be of codimension $k-1$.
It is then enough to show that $\omega(\gamma) \cap G$ deformation retracts to $\omega(\gamma) \cap \partial G$.
But this is straightforward from \Cref{lemma:inductive omega}.
\end{proof}

In particular, $\omega(F)$ deformation retracts to $F$ for any $F \in \FA$.
Further, the obvious choice of these deformation retracts is $W$-equivariant.
That is, we can pick the deformation retracts $r^s_F \from \omega(F) \to \omega(F)$, $s \in [0,1]$ such that for any $w \in W$, $w \circ r^s_F = r^s_{wF} \circ w$.
In particular, letting $r_F = r_F^1 \from \omega(F) \to F$, $r_{wF}\circ w = w \circ r_F$.

\begin{lemma}\label{lemma:omega equivariant}
Let $w \in W$ and $F \in \FA$ be such that $(w \cdot \omega(F)) \cap \omega(F) \ne \emptyset$.
Then $wF = F$.
\end{lemma}
\begin{proof}
First, $\gamma \in \Chain(F)$ iff $w\gamma \in \Chain(wF)$.
So, by \Cref{lemma:theta equivariant}, we have $w\omega(\gamma) = \omega(w\gamma)$.
Now by \Cref{lemma:omegas intersect then preceq}, either $F \le wF$ or vice versa.
But since $F$ and $wF$ are of the same dimension, $F = wF$.
\end{proof}

To construct $U(F,C)$, fix the face $F$.
By the local smoothness assumption, since $F$ is contractible, $TX|_{\ol{F}}$ is $W_F$-equivariantly isomorphic to $\ol{F} \times \R^n$ with $W_F$ acting on $\R^n$ as a linear reflection group.
We will identify the two bundles by this isomorphism.

For a point $x\in F$ the tangent space $T_xF$ has an intrinsically defined convex cone comprising of the `inner' tangent vectors, this is called the \emph{inner sector} (see \cite[Section~2]{joyce12} and \cite[Section~4.1]{michoretal07}).
Note that there exists a coordinate neighborhood of $x$ on which the restriction of $\A_x$ is a hyperplane arrangement.
The local smoothness assumption tells us that the local arrangement around $x$, under taking inner sectors, is combinatorially isomorphic to the arrangement in $T_xX$.
Let $C^T_F$ be the chamber in $\R^n$ (under the above identification) that is fiberwise given by the inner sector of $C_F$.

Now by definition of $r_F$, $T\omega(F)$ is $W$-equivariantly isomorphic to $r_F^*(TX|_F)$.
Let the corresponding bundle map be $\tilde r_F \from T\omega(F) \to F\times \R^n$ and define $U(F,C)$ to be $\tilde r_F^{-1} (F \times C^T_F)$.
By construction, $U(wF,wC) = wU(F,C)$ for any $w \in W$.

For $x \in F$, denote $\tilde r_F^{-1}(r_F(x) \times C^T_F)\cap T_xX = U(F,C) \cap T_xX$ by $T_x^+C_F$.
This is the inner sector of $C_F$ at $x$.

\begin{lemma}\label{lemma:tangent space inclusion}
Let $(F,C),(G,D) \in \Sal_0(\A)$ and $F \le G$. Let $x \in \omega(F) \cap \omega(G)$. Then $T_x^+C_F \subseteq T_x^+D_G$ iff $C_F \subseteq D_G$. Further, $T_x^+D_G$ is covered by the closures of $T_x^+C_F$ as $C$ varies over the chambers containing $F$ such that $C_F \subseteq D_G$.
\end{lemma}
\begin{proof}
By construction, $r_G(x) \in \omega(F)$, $r_F(x) = r_F(r_G(x))$, and the homotopy between the identity of $\omega(F) \cap \omega(G)$ and $r_F|_{\omega(F) \cap \omega(G)}$ goes through $r_G|_{\omega(F) \cap \omega(G)}$.
So it is enough to show the statement for $r_G(x)$.
That is, we can assume $x \in G \cap \omega(F)$.
In this case, $T_x^+D_G = D^T_G$.

But then, since $W_G \subseteq W_F$, and $r_F$ is equivariant under $W_F$ and in particular $W_G$, we have that the fiber isomorphism $T_xX \to T_{r_F(x)}X$ preserves the chambering by $W_G$.
In particular, for a chamber $C \ge F$, $T_x^+C_F$ intersects $D_G^T$ iff $T_x^+C_F \subseteq D_G^T$.

Now by a connectedness argument on $\omega(F) \cap \ol{G}$, $D_G^T$ contains $T_x^+C_F$ iff it contains $C_F^T$.
Using the definition of $C_F^T$ and $D_G^T$ on the tangent space of a point on $F \subseteq \ol{G}$, we see that this happens iff $C_F \subseteq D_G$.
\end{proof}

\begin{lemma}\label{lemma:U in complement}
Let $(F,C) \in \Sal_0(\A)$.
Then $U(F,C) \subseteq M(\A)$.
\end{lemma}
\begin{proof}
Let $(x,v) \in U(F,C)$.
Let $G \in \FA$ be such that $x \in G$.
Since $x \in \omega(F)$, $F \le G$. 
Suppose for some $w \in W$ that is not identity, $w(x,v) = (x,v)$.
Then $wx = x$ and hence, by \Cref{lemma:omega equivariant}, $w \in W_F$.
But then $r_F w = w r_F$ and hence $\tilde r_F w = w \tilde r_F$.
But this means $\tilde r_F(v)$ is fixed by $w \in W_F$, which contradicts $\tilde r_F(v) \in C^T_F$.
\end{proof}

\begin{lemma}\label{lemma:U determines F C}
Let $(F,C), (G,D) \in \Sal_0(\A)$.
If $U(F,C) = U(G,D)$, then $(F,C) = (G,D)$.
\end{lemma}
\begin{proof}
Since $U(F,C)$ is a bundle over $\omega(F)$, it determines $\omega(F)$.
By construction, $F$ is the unique minimal face contained in $\omega(F)$.
So $F = G$.

So now by \Cref{lemma:tangent space inclusion} we have that $C_F = D_F$.
Suppose for the sake of contradiction that $C \ne D$.
But then in some neighborhood of a point $x \in F$, $C$ and $D$ must be separated by some wall corresponding to a reflection $r \in W_F$.
So $C$ and $D$ are in different components of $X \setminus X_r$.
Since $X_r$ occurs as a wall in $\A_F$, $C_F$ cannot then be equal to $D_F$.
\end{proof}

\begin{lemma}\label{lemma:Us cover}
We have
\[M(\A) = \bigcup_{(F,C) \in \Sal_0(\A)} U(F,C) \dispunct{.}\]
\end{lemma}
\begin{proof}
Let $(x,v) \in M(\A)$.
Let $F \in \FA$ be such that $x \in F$.
Since $(x,v)$ is not fixed by any non-identity element of $W$, $\tilde r_F(v)$ is not fixed by any non-identity element of $W_F$.
Hence, for some $C \in \Ch$, $F \le C$ and $\tilde r_F(v) \in C^T_F$.
Then $(x,v) \in U(F,C)$. 
\end{proof}

\begin{lemma}\label{lemma:Us intersect then preceq}
Let $(F,C), (G,D) \in \Sal_0(\A)$.
If $U(F,C) \cap U(G,D) \ne \emptyset$, then either $(F,C) \preceq (G,D)$ or $(G,D) \preceq (F,C)$.
\end{lemma}
\begin{proof}
We have, since $U(C,F)$ and $U(G,D)$ are bundles over $\omega(F)$ and $\omega(G)$ respectively, that $\omega(F) \cap \omega(G) \ne \emptyset$.
Then by \Cref{lemma:omegas intersect then preceq}, without loss of generality, let $F \le G$.
Now by \Cref{lemma:tangent space inclusion}, we get $C_F \subseteq D_G$.
\end{proof}

By induction, we get:
\begin{lemma}\label{lemma:Us intersect then chain}
Let $(F_0,C_0),\dots,(F_p,C_p) \in \Sal_0(\A)$.
If $U(F_0,C_0) \cap \dots \cap U(F_p,C_p) \ne \emptyset$, then up to a permutation of the indices we have $(F_0,C_0) \preceq \dots \preceq (F_p,C_p)$.
\end{lemma}

\begin{lemma}\label{lemma:Us chain then contractible}
Let $(F_0,C_0),\dots,(F_p,C_p)$ be a chain in $\Sal_0(\A)$.
Then $U(F_0,C_0)\cap \dots \cap U(F_p,C_p)$ is non-empty and contractible.
\end{lemma}

\begin{proof}
By \Cref{lemma:tangent space inclusion} the intersection is a bundle over $\omega(F_0,\dots,F_p)$ with fiber $(C_0)_{F_0}$, which is contractible.
By \Cref{lemma:omega contractible}, the base space is contractible.
\end{proof}

\begin{lemma}\label{lemma:W free on Us}
Let $(F,C) \in \Sal_0(\A)$ and $w \in W \setminus \{1\}$. Then $wU(F,C) \cap U(F,C) = \emptyset$.
\end{lemma}
\begin{proof}
If $w \notin W_F$, we are done by \Cref{lemma:omega equivariant}.
If $w \in W_F$, we are done by \Cref{lemma:tangent space inclusion}.
\end{proof}

\begin{proof}[Proof of \Cref{main theorem}]
We have constructed a family $\set{U(F,C)}{(F,C) \in \Sal_0(\A)}$ of open subsets of $M(\A)$, with the following properties:
\begin{enumerate}
\item The assignment of $U(F,C)$ to $(F,C)$ is one-to-one. This is \Cref{lemma:U determines F C}.
\item The sets $U(F,C)$ cover $M(\A)$. This is \Cref{lemma:Us cover}.
\item For $w \in W$, $wU(F,C) = U(wF,wC)$. This is by construction.
\item For $(F_0,C_0), (F_1,C_1), \dots, (F_p,C_p) \in \Sal_0(\A)$, the intersection
\[U(F_0,C_0) \cap U(F_1, C_1) \cap \cdots \cap U(F_p,C_p)\]
is non-empty if and only if, up to permutation, we have a chain
\[(F_0,C_0) \preceq (F_1,C_1) \preceq \dots \preceq (F_p,C_p)\dispunct{.}\]
For such a chain, the intersection of $U(F_i,C_i)$ is contractible.
This is by \Cref{lemma:Us intersect then chain,lemma:Us chain then contractible}.
\item If $w \in W\setminus\{1\}$, then $wU(F,C) \cap U(F,C) = \emptyset$. This is \Cref{lemma:W free on Us}.
\end{enumerate}

Thus, the nerve of the cover $\{U(F,C)\}$ of $M(\A)$ corresponds equivariantly to the order complex of $(\Sal_0(\A),\preceq)$.
By the equivariant nerve lemma, we are done.
\end{proof}

% Section 4: group theoretic info
\section{The fundamental groups}\label{sec:group}
In this section we show that the presentation for the fundamental group of the orbit space, in most cases, is similar to that of Artin groups.
The results stated in this section are similar to the ones given in \cite[Section 3]{parisarxiv}.
As the proofs are absolutely identical we just mention the references here.

\subsection{Coxeter equipment and the Salvetti complex} \label{subsec:group salvetti}
Let $X$ be a smooth $l$-manifold and $\A$ be a reflection arrangement corresponding to a Coxeter transformation group $(W, S)$.
As before, denote the fundamental chamber by $C$, which we would like to be a simple convex polytope, and its face poset by $\F(C)$.

%Let $P$ be a simple convex polytope of dimension $l$. 
Consider the surjective map $\sigma$ from the mirrors of $\ol{C}$ onto the set of generators $S$ of the Coxeter system $(W, S)$ given by sending $C_s\mapsto s$.
Let $F$ be a $k$-face of $C$, since it is the intersection of $k$ codimension-$1$ faces, say $\{C_{s_1},\dots, C_{s_k} \}$, we can extend the map $\sigma$ from $\F(C)$ to the power set of $S$ as follows:
\[F = C_{s_1}\cap\cdots\cap C_{s_k}\mapsto \{ s_1, \dots, s_k\}.\]
If for each face $F$ the subgroup generated by $\sigma(F)$ is finite then the we call $\sigma$ a \emph{Coxeter equipment} of $C$ by $(W, S)$ (see \cite[Section 4.2]{michoretal07}).
The reader can verify that a Coxeter equipment is injective and an order reversing poset map.

\bd{def1sec4}
A subset $T$ of $S$ is said to be an \emph{acceptable (spherical) subset} if the subgroup $W_T$ is finite and the intersection $\bigcap_{s\in T} C_s$ is a non-empty face of $C$. 
\ed 
We denote by $\s^f_X$ the set of all acceptable subsets.
The following lemmas follow from the universal construction of Vinberg (see \cite{vinberg71}). 

\bl{lem1sec4}
With the notation as before the given Coxeter equipment $\sigma\colon \F(C)\to \s^f_X$ is a bijection.
Moreover the following properties hold.
\begin{enumerate}
\item For $T, U\in \s^f_X$ we have $T\subseteq U \iff \sigma^{-1}(U)\leq \sigma^{-1}(T)$.
\item For $T\in \s^f_X$ the stabilizer of the face $F \vcentcolon= \sigma^{-1}(T)$ is equal to $W_T$, and every element of $W_T$ pointwise fixes $F$.
\end{enumerate}
\el
\bpr \cite[Theorem 2.10]{parisarxiv} \epr

We set $\P^f_X \vcentcolon= \{wW_T\mid w\in W \hbox{~and~} T\in \s^f_X \}$ and order it by inclusion. 
Since every face in $\F(\A)$ is of the form $wF$ for some $w\in W$ and $F\in \F(C)$ one can extend the Coxeter equipment to whole of $\F(\A)$ by sending $wF\mapsto wW_F$.

\bl{lem2sec4}
The above defined extension of the Coxeter equipment gives a bijection between $\P^f_X$ and $\F(\A)$. 
Moreover, the following properties hold.
\begin{enumerate}
\item Let $v, w\in W$ and $T, U\in \s^f_X$. We have $vW_T\subseteq wW_U$ if and only if $\sigma^{-1}(wW_U) \leq \sigma^{-1}(vW_T)$.
\item Let $v, w\in W$ and $T, U\in \s^f_X$. We have $vW_T\subseteq wW_U$ if and only if $T\subseteq U$ and $w^{-1}v\in W_U$.
\item The restriction of the extended $\sigma$ gives the bijection between the chambers and the elements of $W$.
\end{enumerate}
\el

Before stating the next lemma we recall the following definition.

\bd{def2sec4}
Let $(W, S)$ be a Coxeter system and $T$ a subset of $S$.
An element $w\in W$ is said to be \emph{$(\emptyset, T)$-minimal} if any one of the following equivalent condition is satisfied.
\begin{enumerate}
\item The element $w$ is of minimal length in $w W_T$;
\item $\ell(ws) > \ell(w)$ for all $s\in T$;
\item $\ell(wu) = \ell(w) + \ell(u)$ for all $u\in W_T$.
\end{enumerate}
\ed
The following lemma ties the geometry of the arrangement to the combinatorics of $W$.
\bl{lem3sec4}
An element $w\in W$ is $(\emptyset, T)$-minimal for some $T\subseteq S$ if and only if for every reflection $r\in W_T$ the corresponding wall $X_r\notin \mathcal{R}(C, w^{-1}C)$. 
\el

\bl{lem4sec4}
Let $\A$ be a manifold reflection arrangement in $X$ associated with a Coxeter system $(W, S)$ and the Coxeter equipment $\sigma$. 
Then the relation $\preceq$ defined as:
\[(v, T)\preceq(w, U)\iff T\subseteq U, w^{-1}v\in W_U, w^{-1}v \text{ is $(\emptyset, T)$-minimal}\]
is a partial order on $\s^f_X\times W$.
\el
\bpr
\cite[Lemma 3.2]{parisarxiv}.
\epr

Note that the action of $W$ on $\s^f_X\times W$ defined by $w\cdot (T, v) = (T, wv)$ is order preserving.
Let $C'$ be a chamber of the form $wC$ and $F'$ be one of its face. 
Define the map $\Phi\colon \Sal_0(\A)\to \s^f_X\times W$ as follows
\[\Phi(F', C') = \Phi(wF, wC) = (\sigma(F), w)\dispunct{.}\]

\bt{thm1sec4}
The map $\Phi$ is an order reversing poset isomorphism. 
\et 
\bpr
\cite[Theorem 3.3]{parisarxiv}.
\epr
Hence the map $\Phi$ induces a homeomorphism on the geometric realizations of the corresponding posets. 
Moreover, the reader can verify that the map is also $W$-equivariant.
Combining the above result with \Cref{main theorem} we get the following.
\bc{cor1sec4}
The geometric realization of $\s^f_X\times W$ is $W$-equivariantly homotopy equivalent to the associated tangent bundle complement $M_W$.
\ec

\bc{cor2sec4}
The homotopy type of $N_W$ depends on the combinatorial type of the fundamental chamber and the chosen Coxeter equipment.
\ec

\subsection{A cell structure}
Let $\A$ be a reflection arrangement of submanifolds in an $l$-manifold $X$ corresponding to a Coxeter transformation group $W$.
Let $\F(\A)$ denote the associated face poset. 
By $(X, \F(\A))$ we mean the regular cell structure of $X$ induced by the arrangement. 

We also have the dual cell structure as follows.
For every face $F$ fix a point $x(F)\in F$, say the \emph{barycenter} of $F$. 
Note that $\ol{F}$ is homeomorphic to an appropriate-dimensional disc $B_F$ and carries a regular cell structure given by faces $F' \le F$.	
For every $G < F$ the barycenter $x(G)$ corresponds to a point $y_G$ of $B_F$.
If $\gamma = (G_0 ,\dots, G_k)$ is a chain of faces (with $G_0 < \dots < G_k$) of $F$ then denote by $\gamma_B$ the simplex that is the convex hull of the vertices $y_{G_0}, \dots, y_{G_k}$. 
Let $\Delta(\gamma)$ be the image of $\gamma_B$ under the chosen homeomorphism. 
Let $F^*$ be the union of all those $\Delta(\gamma)$'s that arise from chains beginning in $F$ and call it the \emph{dual cell} of $F$.
The collection of all the dual cells defines a regular cell structure since the link of each vertex is a sphere. 
If $F$ is $k$-dimensional, then $F^*$ is $(l-k)$-dimensional.
The poset $\F^*(\A)$ formed by all the dual faces is the dual of the face poset and regular cell structure $(X, \F^*(\A))$ is called the \emph{dual cell structure}. 

For the sake of notational simplicity, we denote the dual, regular cell complex by $X^*(\A)$ ($X^*$ if the context is clear). 
The symbols $C, D$ will denote vertices of $X^*$ and the symbol $F^k$ will denote a $k$-cell dual to the codimension $k$-face of $\A$. 
Note that a $0$-cell $C$ is a vertex of a $k$-cell $F^k$ in $X^*$ if and only if the closure of the chamber corresponding to $C$ contains the $(l-k)$-face dual to $F^k$. 
The action of the faces on chambers that was introduced in \Cref{def3sec3} also holds true for the dual cells. 
The symbol $F^k\circ C$ will denote the vertex of $F^k$ which is dual to the unique chamber `closest' to the chamber corresponding to $C$.  
The partial order on the cells of $X^*$ will be denoted by $\sqsubseteq$.
Now given $X^*(\A)$ construct a CW complex $S(\A)$ of dimension $l$ as follows:

The $0$-cells of $S(\A)$ correspond to $0$-cells of $X^*$, which we denote by the pairs $\left<C; C\right>$. 
For each $1$-cell $F^1 \in X^*$ with vertices $C_1, C_2$, take two homeomorphic copies of $F^1$ denoted by $\left<F^1; C_1\right>, \left<F^1; C_2\right>$. 
Attach these two $1$-cells in $S(\A)_0$ so that 
\[\partial \left<F^1; C_i\right> = \{\left<C_1; C_1\right>, \left<C_2; C_2\right> \}.\]
Orient the $1$-cell $\left<F^1; C_i\right>$ so that it begins at $\left<C_i; C_i\right>$, to obtain an oriented $1$-skeleton $S(\A)_1$.

By induction assume that we have constructed the $(k-1)$-skeleton of $S(\A)$, $1\leq k-1 < l$. 
To each $k$-cell $F^k \in X^*$ and to each of its vertex $C$ assign a $k$-cell $\left<F^k; C\right>$ that is isomorphic to $F^k$.
Let $\phi(F^k, C)\colon \left<F^k; C\right>\to S(\A)_{k-1}$ be the same characteristic map that identifies a $(k-1)$-cell $F^{k-1}\subset \partial F^k$ with the $k$-cell $\left<F^{k-1}; F^{k-1}\circ C\right>\subset \partial \left<F^k; C\right>$.
Extend the map $\phi(F^k, C)$ to whole of $\left<F^k; C\right>$ and use it as the attaching map, hence obtaining the $k$-skeleton. The boundary of every $k$-cell is given by 
\begin{equation}
\partial \left<F^k; C\right> = \bigcup_{F \sqsubseteq F^k} \left<F; F\circ C\right>. \label{eq1s3c3}
\end{equation}

The following lemma is now clear
\bl{lem5sec4}
The cell complex $S(\A)$ constructed above has the same homotopy type as that of the associated tangent bundle complement $M_W$.
\el
\bpr
This follows from the fact that the face poset of $S(\A)$ is isomorphic to the poset $(\Sal_0(\A),\preceq)$. 
Moreover, in view of \Cref{thm1sec4} it is also isomorphic to $\s^f_X\times W$.
\epr

\bt{thm2sec4} Let $\A$ be a manifold reflection arrangement in an $l$-manifold $X$ corresponding to $W$ and let $S(\A)$ denote the associated Salvetti cell complex. 
\begin{enumerate}
	\item There is a natural cellular map $\psi\colon S(\A)\to \F^*(\A)$ given by $\mleft< F, C\mright>\mapsto F$. The restriction of $\psi$ to the $0$-skeleton is a bijection and in general 
	\[ \psi^{-1}(F) = \{C\in \Ch \mid F \leq C \}.\] 
	\item  For every chamber $C$ there is a cellular map $\iota_C\colon \F^*(\A)\to S(\A)$ taking $F$ to $\mleft<F, F\circ C\mright>$ which is an embedding of $\F^*(\A)$ into $Sal(\A)$, and \[S(\A) = \bigcup_{C\in \Ch} \iota_C(\F^*).\] 
	\item If $W$ is finite then the absolute value of the Euler characteristic of $M(\A)$ is $\lvert W\rvert$.
    \item Let $T\A$ denote the union of the tangent bundles of the submanifolds in $\A$. Then $\widetilde{H}_i(TX, T\A) = 0$ unless $i = l$.
\end{enumerate} \et

\begin{proof}
(1) and (2) follow from the simple observation that $\psi\circ\iota_C$ is identity on $\F^*(\A)$. 
This also implies that $X$ is homeomorphic to a retract of $M(\A)$.\par  
We prove (3) by explicitly counting cells in the Salvetti complex. 
The Euler characteristic of a cell complex $K$ is equal to the alternating sum of number of cells of each dimension. 
Given a $k$-dimensional dual cell $F$ there are as many as $|\{C\in\Ch \mid F\leq C \}|$ $k$-cells in $\psi^{-1}(F)$. 
Hence for a vertex $\mleft<C, C\mright>\in S(\A)$ the number of $k$-cells that have $\mleft<C, C\mright>$ as a vertex is equal to the number of $k$-faces contained $C$. 
The alternating sum of numbers of such cells is $1 - \chi(\Lk(C))$, where $\Lk(C)$ is the link of $C$ in $\F^*(\A)$. 
Applying this we get 
\[\chi(S(\A)) = \sum_{C\in \Ch}(1 - \chi(\Lk(C))). \] 
Since each chamber is bounded, $\Lk(C)\simeq S^{l-1}$. Hence we have, 
\begin{align*}
		\chi(S(\A)) &= \sum_{C\in \Ch}(1 - \chi(\Lk(C)))  \\
		              &= \sum_{C \in \Ch}(1 - [1 + (-1)^{l-1}]) \\
					  &= (-1)^l \sum_{C \in \Ch}1.
\end{align*}
Hence, \[ \chi(M(\A)) = (-1)^l \hbox{(number of chambers)} \dispunct{.} \]
Note that there is a one-to-one correspondence between elements of $W$ and chambers in $\Ch$.

Let $\bigcup\A$ denote the union of submanifolds in $\A$. 
Since $\A$ induces a regular cell decomposition it has the homotopy type of wedge of $(l-1)$-dimensional spheres indexed by chambers. 
Claim (4) now follows from the homeomorphism of pairs $(TX, T\A) \cong (X, \bigcup\A )$.
\end{proof}

In the rest of this section we identify the cells of $S(\A)$ with the elements of $\s^f_X\times W$. 
For example, a $k$-cell is represented by $\langle T, w\rangle$ where $|T| = k$ and $w\in W$.
The $W$-action on $S(\A)$ defined by $w'\cdot \langle T, w\rangle = \langle T, w'w\rangle$ is cellular. 
The orbit complex denoted by $\ol{S}(\A)$ has the homotopy type of $N_W$. 
The cellular decomposition of $S(\A)$ determines a cellular decomposition of $\ol{S}(\A)$ which we now describe.
The orbit of the $k$-cell $\langle T, 1\rangle$ is $\{\langle T, u\rangle\mid |T| = k, w\in W \}$.
Hence $k$-cells of the orbit complex are determined by the $k$-subsets of $\s^f_X$; we denote these cells by $[T]$.
Note that the cell $[T]$ is homeomorphic to $[T, 1]$ (which in turn is homeomorphic to the cell dual to the face fixed by $W_T$) via a homeomorphism $h_T$.
The attaching map $\ol{\phi}_T\colon \partial [T]\to \ol{S}(\A)_{k-1}$ is defined as follows
\[ \ol{\phi}_T = \pi\circ \phi(T, 1)\circ h_T,\]
where $\pi\colon S(\A)\to \ol{S}(\A)$ is the natural projection.
Note that $\ol{\phi}_T$ is not in general a homeomorphism and the closed cell $[T]$ need not be embedded in $\ol{S}(\A)$.

In order to identify a presentation for $\pi_1(N_W)$ we describe the $2$-skeleton of these complexes.

\noindent \textbf{$0$-skeleton}: As described above the vertex set of $S(\A)$ is $\{\langle\emptyset, w \rangle \mid w\in W\}$.
For notational simplicity we denote a vertex of $S(\A)$ by $x(w)$. 
There is only vertex $x_0$ in the orbit complex.

\noindent \textbf{$1$-skeleton}: Edges in $S(\A)$ are determined by the pairs consisting of elements of $S$ and elements of $W$.
We denote such an edge by the symbol $e(s, w)$; its vertices are $x(w)$ and $x(ws)$.
We orient the edges such that their initial vertex is $x(w)$.
The construction implies that there is another edge labeled $e(s, ws)$ which starts at $x(ws)$ and ends at $x(w)$.
Whereas in the orbit complex there is a unique edge $\ol{e}_s$ with both the end points being $x_0$.
Note that the $W$-action preserves the orientation on $S(\A)_1$ hence it induces an orientation on $\ol{e}_s$'s.

Before explaining the $2$-skeleton we adopt a notational convenience. 
For $m>0$ denote the product $\underbrace{sts\cdots}_m$ by the symbol $\ts\prod(s,t:m)$.

\noindent \textbf{$2$-skeleton}: Let $s, t\in S$ such that $T := \{ s, t\}\in \s^f_X$ and let $m := m_{st}$ denote the order of $st$. For every $w\in W$ there is a $2$-cell $\langle T, w \rangle$ in $S(\A)$ whose boundary is
\[e(s, w) e(t, ws) e(s, wst)\cdots e(t, w\ts\prod(s, t:m-1)) e(s, w\ts\prod(t, s:m-1))^{-1}\cdots e(s, wt)^{-1} e(t, w)^{-1} \]
if $m$ is even, and
\[e(s, w) e(t, ws) e(s, wst)\cdots e(s, w\ts\prod(s, t:m-1)) e(t, w\ts\prod(t, s:m-1))^{-1}\cdots e(s, wt,)^{-1} e(t, w)^{-1} \]
if $m$ is odd. The $W$-orbit of $2$-cells $\{\langle T, w\rangle\mid w\in W \}$ determines the $2$-cell $[T]$ of $\ol{S}(\A)$ whose boundary curve is given by 
\[\ts\prod(\ol{e}_s, \ol{e}_t: m)\ts\prod(\ol{e}_t, \ol{e}_s: m)^{-1}. \]
Note that this expression is independent of the parity of $m$ since there is only vertex $x_0$.

From the above discussion following conclusion can be drawn.

\bt{thm3sec4}
Let $\A_W$ be a manifold reflection arrangement in $X$ corresponding to the Coxeter transformation group $W$. The fundamental group of the orbit space $N_W$ has the following presentation 
\[\pi_1(N_W) = \langle \ol{e}_s \mid s\in S, \ts\prod(\ol{e}_s, \ol{e}_t: m_{st}) = \ts\prod(\ol{e}_s, \ol{e}_t: m_{st}) \hbox{~for every~} \{s, t\}\in \s^f_X \rangle. \]
\et

\bt{thm4sec4}
If $\A_W$ is manifold reflection arrangement in a simply connected manifold $X$ then $\pi_1(N_W)$ is the Artin group associated with $W$.
\et 

\bpr
According to \cite[Theorem 9.1.3]{davisbook08} if the manifold $X$ is simply connected then for every $2$-element spherical subset $\{ s, t\}$ the corresponding intersection of mirrors $C_s\cap C_t$ is non-empty.
Which means that the subset is acceptable and there is a $2$-cell in $\ol{\Sal}(\A_W)$ corresponding to this subset.
\epr

If $W$ is a finite reflection group (of rank at least $3$) acting on a simply connected sphere then $\pi_1(N_W)$ is the Artin group of finite type. 
However note that the Salvetti complex constructed above is not a model for the $K(\pi, 1)$-space.
The reader can verify, using \Cref{thm2sec4}, that it has non-trivial higher homotopy groups. 
Finally, we have some examples.

\be{ex1sec4} 
Recall the action of the dihedral group on $S^1$ from \Cref{ex1}. 
This action gives an arrangement of $2m$ points on $S^1$.
The tangent bundle complement of this arrangement is an infinite cylinder with $2m$ punctures.
Hence it has the same homotopy type as that of wedge of $2m + 1$ circles. 
As for the quotient complex since there are only two non-empty acceptable subsets each of cardinality $1$ it has the homotopy type of wedge of two circles.
\Cref{s3ons1sal} illustrates the situation for $S_3$ action.
\begin{figure}[!ht]
  \begin{center}  
    \includegraphics[scale=0.5,clip]{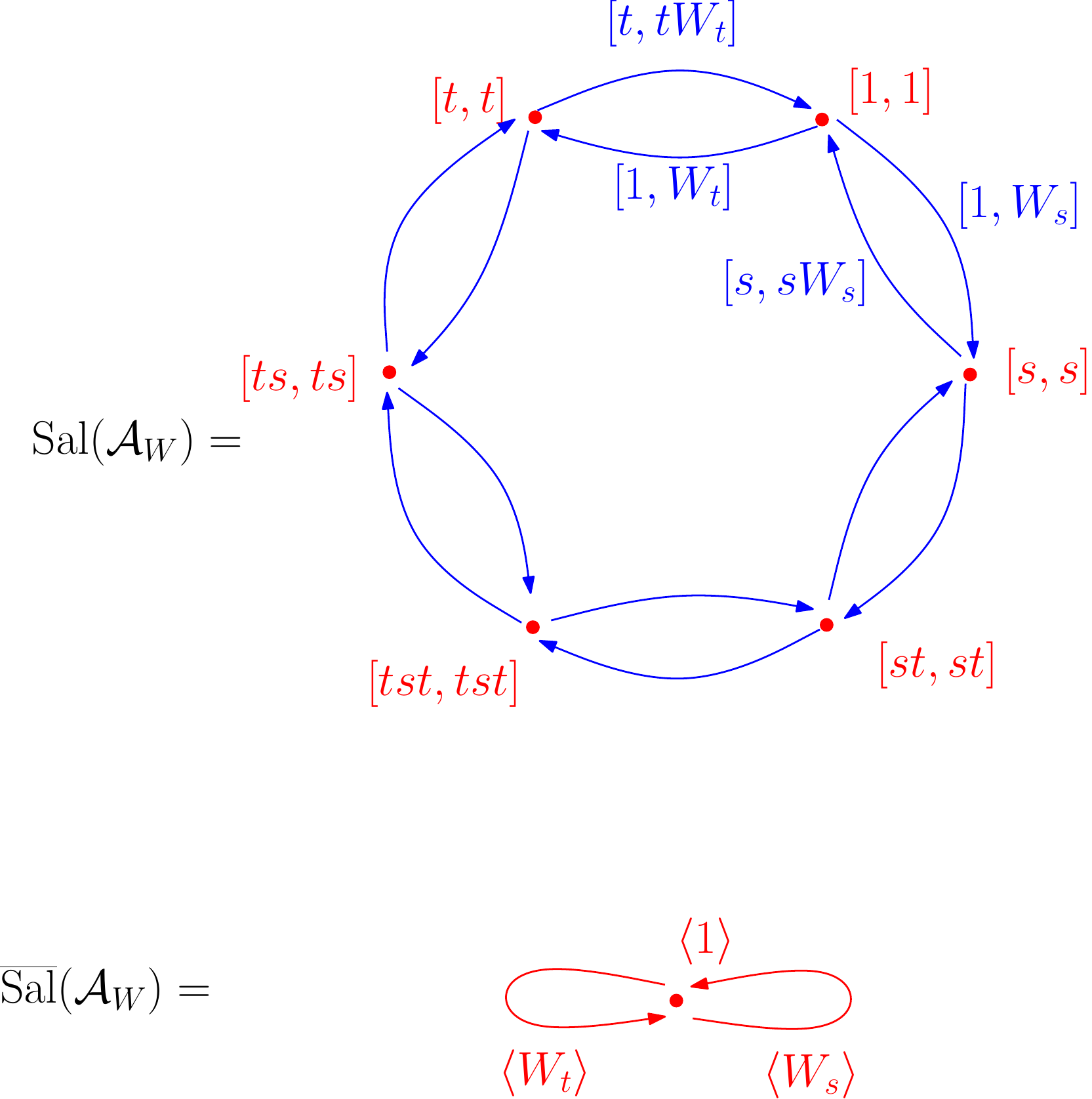} 
  \end{center}
  \caption{The Salvetti complex.}     \label{s3ons1sal}     
\end{figure}
\ee

\be{ex2sec4}
Let $W_a$ be an irreducible, affine Weyl group (such groups are classified by the extended Coxeter-Dynkin diagrams). 
Then $W_a$ is the semidirect product of subgroup of translations (i.e., the rank $l$ coroot lattice) and a finite reflection group $W$ (see \cite[Chapter 4]{humph90}).
As a Coxeter group $W_a$ is generated by $l+1$ reflections $\{s_0, s_1\dots, s_l \}$.
Reflections from $s_1$ to $s_l$ generate $W$ and $s_0$ is the reflection corresponding to the translate of the longest root. 
The $W$-action on $\R^l$ fixes a finite family of parallel hyperplanes. 
The fundamental domain of this action is an alcove (an $l$-simplex to be precise). 
If $T$ is any subset of the generators of cardinality at most $l$ then the corresponding parabolic subgroup is finite.
Hence any subset of cardinality at most $l$ is spherical.
In view of \cite[Theorem 9.1.4]{davisbook08}, in this case, every spherical subset is acceptable.
As a result the Salvetti complex construction described in this section coincides with the one given by Charney and Davis (see \Cref{intro}).
We conclude that the tangent bundle complement associated with affine Weyl groups has the same homotopy type as the Vinberg complement, i.e., the space defined in \Cref{vinmw}.
\ee

\br{remark2}
As in the classical case one can also construct the universal cover of the Salvetti complex in this context. 
Since $M_W\to N_W$ is a regular cover we have the following exact sequence of groups
\[1\to \pi_1(M_W)\to \pi_1(N_W)\to W\to 1. \]
Choosing a suitable section of the map $\pi_1(N_W) \to W$ one can define a partial order on the set $\pi_1(N_w)\times \s^f_X$ such that the geometric realization is the universal cover (see \cite[Section 4.1]{parisarxiv} for details).
\er 

\br{remark3}
There is one more way to generalize the Vinberg complement. Given a reflection arrangement $\A_W$ in a smooth manifold $X$ consider the space
\[ B_W := (X\times X)\setminus \bigcup_{r\in R}(X_r\times X_r).\]
One might call $B_W$ the \textit{micro-bundle complement} because of its resemblance with Milnor's tangent micro-bundle construction.
Consider the dihedral group action on $S^1$.
In this case the micro-bundle complement $B_W$ is the torus $S^1\times S^1$ minus $4m$ points.
It is not hard to see that $B_W$ has the homotopy type of wedge of $4m+1$ circles which is not the homotopy type of associated $M_W$.
Hence it is an interesting problem to figure out a relationship between the homotopy types of $B_W$ and $M_W$. 
\er 

\printbibliography

\end{document}